\documentclass[12pt]{article}
\usepackage[utf8]{inputenc}  \usepackage[T1]{fontenc}  \usepackage{lmodern}
\usepackage{array} 
\usepackage{paralist} 
\usepackage{amsmath} \usepackage{amsthm} 
\usepackage[british]{babel}  \usepackage{amssymb}  \usepackage{graphicx}
\usepackage[left=2.5cm,top=2.5cm,right=2.5cm,bottom=2cm]{geometry}  \usepackage{color} 
\usepackage{comment}

\newtheorem{theorem}{Theorem}[section]
\newtheorem{lemma}[theorem]{Lemma}

\newtheorem{proposition}[theorem]{Proposition}
\newtheorem{definition}[theorem]{Definition}
\theoremstyle{definition}
\newtheorem{remark}[theorem]{Remark}

\newtheorem{notation}[theorem]{Notation} 

\numberwithin{equation}{section}
\newcommand{\R}{\mathbb{R}}

\newcommand{\Z}{\mathbb{Z}}
\newcommand{\C}{\mathbb{C}} 
\newcommand{\N}{\mathbb{N}} 

\newcommand{\MS}{\mathcal{S}}
\newcommand{\MT}{\mathcal{T}}
\newcommand{\MJ}{\mathcal{J}}

\newcommand{\MX}{\mathcal{X}}
\newcommand{\MY}{\mathcal{Y}}

\newcommand{\MN}{\mathcal{N}}
\newcommand{\MK}{\mathcal{K}}

\newcommand{\MG}{\mathcal{G}}
\newcommand{\ME}{\mathcal{E}}
\newcommand{\vp}{\varphi}
\newcommand{\la}{\lambda}

\usepackage{graphicx}
\usepackage{booktabs}
\usepackage{tikz}
\usepackage{geometry}
\usepackage{pgfplots}
\usepackage{etex}
\usetikzlibrary{arrows,shapes,positioning}
\usetikzlibrary{decorations.markings}
\tikzstyle arrowstyle=[scale=1]
\tikzstyle directed=[postaction={decorate,decoration={markings,
		mark=at position .65 with {\arrow[arrowstyle]{stealth}}}}]
\tikzstyle directeddd=[postaction={decorate,decoration={markings,
		mark=at position .85 with {\arrow[arrowstyle]{stealth}}}}]
\tikzstyle reverse directed=[postaction={decorate,decoration={markings,
		mark=at position .65 with {\arrowreversed[arrowstyle]{stealth};}}}]
\tikzstyle reverse directeddd=[postaction={decorate,decoration={markings,
		mark=at position .85 with {\arrowreversed[arrowstyle]{stealth};}}}]
\usepackage{pgfplots}
\usepackage[most]{tcolorbox}

\title{Differentiable invariant manifolds of nilpotent parabolic points}

\date{}

\begin{document}
\maketitle

\vspace{-1.5cm}

\centerline{Clara Cufí-Cabré}
\medskip
{\footnotesize
	\centerline{Departament de Matemàtiques} 
	\centerline{Universitat Autònoma de Barcelona (UAB)} 
	\centerline{Barcelona Graduate School of Mathematics (BGSMath)}
	\centerline{08193 Bellaterra, Barcelona, Spain}
	\centerline{\texttt{clara@mat.uab.cat}}
} 

\medskip
\vspace{8pt}

\centerline{Ernest Fontich}
\medskip
{\footnotesize
	\centerline{Departament de Matemàtiques i Informàtica} 
	\centerline{Universitat de Barcelona (UB)}
	\centerline{Barcelona Graduate School of Mathematics (BGSMath)}
	\centerline{Gran Via 585. 08007 Barcelona, Spain}
	\centerline{\texttt{fontich@ub.edu}}
}

\bigskip

\begin{abstract}
	We consider a map $F$ of class $C^r$ with a fixed point of parabolic type whose differential is not diagonalizable and we study the existence and regularity of the invariant manifolds associated with the fixed point using the parameterization method. Concretely, we show that under suitable conditions on the coefficients of $F$, there exist invariant curves of class $C^r$ away from the fixed point, and that they are analytic when $F$ is analytic. The differentiability result is obtained as an application of the fiber contraction theorem. We also provide an algorithm to compute an approximation of a parameterization of the invariant curves and a normal form of the restricted dynamics of $F$ on them. 
\end{abstract}

\vspace{0.1cm}

\begin{footnotesize}
\end{footnotesize}


\section{Introduction}

Invariant manifolds play a central role in the study of dynamical systems. There is a huge amount of literature devoted to study them in many different settings. In this paper we deal with the invariant manifolds of a type of parabolic fixed points in dimension two.

Parabolic points appear generically in two-parameter families of planar maps or in one-parameter ones in the case of area-preserving maps. 
In particular they appear when a family of maps undergoes a Bogdanov-Takens bifurcation \cite {bog75, takens74}.

In some problems in Celestial Mechanics it is useful to consider parabolic points or parabolic orbits at infinity  in order to use their invariant manifolds (provided they exist) to study features of the dynamics  in the finite phase space. The local study in a neighborhood of such points is done by means of a change of variables which sends the infinity to a finite part of the space \cite{mg73}. Also, the periodic orbits become fixed points of appropriate (families of) Poincar\'e maps.
In such cases the fixed points are parabolic for all values of the parameters of the family and may have invariant manifolds. These manifolds have been used to prove the existence of oscillatory motions in the Sitnikov problem \cite{sit60,mos73} and the restricted planar three-body problem
\cite{llibresimo80,gms16,gsms17} using the transversal intersection of invariant manifolds of parabolic points and symbolic dynamics. Parabolic manifolds also appear in the Manev problem \cite{cdlp99}.

Parabolic periodic orbits at infinity have been found in Hamiltonian systems related to the study of the scattering of He atoms off Cu surfaces with some corrugation \cite{gbm97}.
These manifolds also play a significant role in the study of certain systems \cite{lm16,gaw20}.

In this paper we consider two-dimensional maps having a parabolic fixed point whose linearization does not diagonalize, concretly we assume it has a double eigenvalue equal to 1. By simple changes such maps can be brought to the form 
\begin{align} 
\label{forma-map}
F(x, y) &= 
\begin{pmatrix}
x + c  y + f_1(x,y) \\
 y + f_2(x,y)
\end{pmatrix}  ,
\end{align}
with $c>0 $, $f_1(0,0) = f_2(0,0)=0$ and $Df_1(0,0)=Df_2(0,0)=0$.
The origin has a center manifold of dimension two, however, inside this manifold there may exist curves that behave topologically as stable or unstable curves.

This class of maps was considered in \cite{fontich99} and the existence of analytic curves was proved. Concretely the (local) sets considered there and the ones we deal with are
\[
W^{s+}_{\rho} = 
\{ (x,y)\mid \,  F^n(x,y) \in (0,\rho) \times ( - \rho, \rho), \, \forall  \, n\ge 0,\, 
\lim_{n\to \infty} F^n(x,y)=0 \} 
\]
and 
$$ W^{u+}_\rho =
\{ (x,y)\mid \,  F^{-n} (x,y) \in (0,\rho) \times ( - \rho, \rho), \, \forall \, n\ge 0,\, 
\lim_{n\to \infty} F^{-n}(x,y)=0 \} .
$$
The main result  of \cite{fontich99} concerns analytic stable invariant curves in
the domain 
$\{(x,y) \in\R^2 \mid \, x\ge 0,\, y\le 0\}$ under some appropriate conditions on the higher order terms. Then, the existence of both stable and unstable curves in neighborhoods of the origin are deduced from the main result by 
using the symmetries $(x,y) \mapsto (-x,y)$, $(x,y)\mapsto (x,-y)$ and $(x,y)\mapsto (-x,-y)$ and the inverse map $F^{-1}$. Moreover, a detailed study of the local dynamics provide the uniqueness of such curves in the category of 
$C^k$ maps where $k$ is the minimum regularity for having a Taylor expansion providing the relevant nonlinear terms \cite{fontich99}. 

In this paper we study the existence and regularity  of stable curves in the domain 
$\{(x,y) \in\R^2 \mid \, x\ge 0,\, y\le 0\}$ using the parameterization method. In the analytic case we recover the existence results of \cite{fontich99} but we also provide approximations of the curves up to an arbitrarily high order. We consider three cases of maps of the form \eqref{forma-map}, already introduced in \cite{fontich99}, which depend in some sense on the dominant part of the nonlinear terms. The study depends on each case. Moreover, we consider the differentiable case with the same method and we obtain that the invariant manifolds of $F$ are of the same regularity as $F$ provided some minimum regularity holds.   
Contrary to other works we  do not use the Poincar\'e normal form for the map, but a simple and easy-to-compute reduced form.

This class of maps, assuming the fixed point is not isolated, was studied in 
\cite{cfn92} motivated by the study of collisions in two-body problems with
central force potential satisfying certain asympotic properties at the origin.
A special case of this family not previously covered is studied in \cite{jm18}. 
 These papers use an adapted form of the method of McGehee for  parabolic points without nilpotent part \cite{mg73}. McGehee's method consists of looking for a sector-like domain $S$, with the fixed point in the vertex, such that the points whose positive iterates remain on $S$ form a graph of some function $\varphi$. To prove analyticity, it considers the complexified map and uses Rouch\'e's theorem to obtain the uniqueness of $\varphi(x)$ in terms of $x$, for $x$ in a complex extension $\overline S$ of $S$, so that then one can apply the implicit function theorem to obtain the analyticity of $\varphi(x)$ for $x \in \overline S$.

 Again for maps of the form \eqref{forma-map}, using different tools, some regularity results are obtained in \cite{zhang16}. In that paper, the authors deal with what we denote by case $1$ for $C^\infty$ maps and obtain the existence of a stable manifold $W_{\rho}^{s+}$  as the graph of some function $\varphi$ by solving a fixed point equation equivalent to the invariance of the graph of $\varphi$. This equation is considered for functions $\varphi$ in a suitable subset of the space of functions of class  $C^{[(k+1)/2]}$, where $[\cdot]$ denotes integer part, and it is solved applying the Schauder fixed point theorem. Hence, they obtain invariant manifolds of class $C^{[(k+1)/2]}$. Instead, in this paper, we use the parameterization method (see Section \ref{sec_param_met}) and we obtain, away from the fixed point, analytic invariant manifolds for analytic maps and $C^r$ invariant manifolds for $C^r$ maps, provided $r$ is larger than some quantity that depends on the nonlinear terms of the map.

One-dimensional manifolds of fixed points with linear part equal to the identity are studied in \cite{bflm07} using the parameterization method.
Higher-dimensional manifolds in the same setting are considered in 
\cite{bf04} using a generalized version of the method of McGehee, and in 
\cite{BFM20a, BFM20b} using the parameterization method, where applications to Celestial Mechanics are given. The Gevrey character of one-dimensional manifolds is studied in \cite{BFM17}.

The main results of this paper are Theorems \ref{teorema_analitic} and \ref{teorema_analitic-a-posteriori}, concerning the existence of analytic invariant curves of a map $F$ of the form \eqref{forma-map}, and Theorems \ref{teorema_cr} and \ref{teorema_cr-a-posteriori}, concerning the existence of differentiable invariant curves. In Section \ref{sec_statement} we present the parameterization method and the main results of the paper. The results are stated for the stable curves. In Section \ref{subsec_inestable} we show that completely analogous results hold true for the unstable ones. 
 In Section \ref{sec_algoritme} we provide an algorithm to obtain parameterizations of approximations of the invariant curves of $F$, and we provide the existence of such curves in Sections \ref{sec_analytic}, for the analytic case, and in Section \ref{sec_cr}, for the differentiable case. The proofs of the technical results used along the paper are deferred to Section \ref{sec_dem_lemes}. The paper finishes with a conclusions section where we summarize the results of the paper.

\section{Statement of the main results} \label{sec_statement}

\subsection{Reduction of the maps to a simple form} \label{sec:notation}
In this paper we consider $C^r$, $r \geq 3$, or analytic maps 
$F: U \subset \R^2 \to \R^2$, where $U$   is a neighborhood of $(0,0)$, of the form
\begin{equation} \label{forma_general}
F(x, y) = 
\begin{pmatrix}
x + c \,  y  + f_1(x, y)\\
y + f_2(x, y)
\end{pmatrix},
\end{equation}
with $c> 0$  and with $f_1(x, y), \, f_2(x, y) = O(\|(x, y)\|^2)$. Via the $C^r$ change of variables given by $\tilde x = x$,   $\tilde{y} = y + \frac{1}{c} \, f_1(x,  y)$,  $F$ can be written in the form
\begin{equation*} \label{forma_unaeq}
F(x, y) = 
\begin{pmatrix}
x + c\,  y  \\
y + f(x,  y)
\end{pmatrix},
\end{equation*}
with $f(x,   y)= O(\|(x,   y)\|^2)$ having the same regularity as $F$. In the
 $C^r$ case we denote by $P(x,  y)$ the Taylor polynomial of degree $r$ of $f(x,  y)$. We write $P(x, y)$ in the form 
 $$
 P(x,  y) = p(x) + y q(x) + u(x,  y),
 $$
 where we have collected all the terms independent of $y$ in $p(x)$, the terms that are linear in $y$ in $yq(x)$ and all remaining terms in $u(x, y)$. Note that all terms in $u(x, y)$ have the factor $y^2$. More precisely, we write $p(x) =  x^k(a_k + \cdots + a_rx^{r-k}) $ and $ q(x) = x^{l-1}(b_l+ \cdots + b_r x^{r-l})$, with $2\le k, l \leq r$.
Therefore we have $ f(x,  y) = P(x,  y) + g(x,  y)$ with $g(x,  y)= o(\|(x,  y)\|)^r$.

Also, note that one can always assume that $c > 0$. If this is not the case, then it can be attained via the linear transformation given by $L(x,  y) = (x,  -y)$,  taking the conjugate map $\tilde{F} = L^{-1} \circ  F \circ L$. 
Notice however that $L$ sends the lower semi plane to the upper one. 
Hence, any map $F$  of  the form \eqref{forma_unaeq}  can be written in the form
\begin{equation} \label{forma_normal_cr}
\bar{F}(x, y) = 
\begin{pmatrix}
x + c\, y \\  
y+ p(x) + yq(x) + u(x,  y) +g(x,  y) 
\end{pmatrix},
\end{equation}
with $c >{\tiny } 0$.
In the analytic case we have the same form with  $g(x, y)$ analytic. In general we will not write the dependence of $p$, $q$, $u$ and $g$ on $r$.
Throughout the paper we will refer to \eqref{forma_normal_cr} as the \emph{reduced form} of $F$ and we will use the same notation $F$. 

We will deal with maps of the form \eqref{forma_normal_cr}. We remark that in contrast with other references \cite{fontich99, zhang16} in which they work with normal forms of $F$ \emph{\`a la Poincar\'e}, we work with the reduced form obtained with a simple change of variables. This is an important advantage when one has to perform effective computations.  

Following \cite{fontich99}, 
we shall consider three cases depending on the indices $k$ and $l$: 
\begin{itemize}
	\item Case $1$: $k <  2l-1$ and $a_k \neq 0$,
	\item Case $2$: $k  = 2l-1$ and $a_k, \, b_l \neq 0$,
	\item Case $3$: $k > 2l-1$ and $b_l \,  \neq 0$.
\end{itemize}
In order to deal, whenever possible, with several cases at the same time we associate to $F$ the integers $N$ and $s$: $N=k$ in case 1 and $N=l$ in cases 2 and 3; $s=2r$ in case 1 and $s=r$ in cases 2, 3.
Notice that the generic case is case 1 with $k=2$.

Next we make a comment concerning notation. The superindices $x$ and $y$ on the symbol of a function or an operator that takes values in $\R^2$ will denote  the first and second components of its image, respectively. In $\R^2$ and $\C^2$ we will use the norm given by $\|(x, y)\| = \max \, \{|x|, \, |y|\} $. Throughout the paper, $M$ and $\rho_0$ will denote positive constants, and they do not take necessarily the same value at different places.

\subsection{The parameterization method} \label{sec_param_met}
To study the stable curves of $F$ we will use the parameterization method (see
\cite{cfdlll03I}, \cite{cfdlll03II}, \cite{cfdlll05}, \cite{HCFLM16}). It consists in looking for the curves as images of parameterizations, $K$, together with a representation of the dynamics of the map restricted to them, $R$, satisfying the invariance equation,
\begin{equation} \label{semiconjugation}
F \circ K  =K \circ R .
\end{equation}
This is a functional equation that has to be adapted to the setting of the problem at hand. Clearly, we need the range of $R$ to be contained in the domain of $K$.
It follows immediately from \eqref{semiconjugation} that the range of $K$ is invariant.
Essentially, $K$ is a (semi)conjugation of the map restricted to the range of $K$ to $R$. Equation  \eqref{semiconjugation} has to be solved in a suitable space of functions. Usually it is convenient to have good approximations of $K$ and $R$
and look for a (small) correction of $K$, in some sense, while maintaining $R$ fixed.
Assuming differentiability and taking derivatives in \eqref{semiconjugation} we get 
$DF \circ K \cdot DK  =DK \circ R \cdot DR$ which says that the range of $DK$ has to be invariant by $DF$.

In our setting we look for $K=(K^x, K^y):[0,\rho) \to \R^2 $ such that $K(0)=(0,0)$ and $DK(t)$ satisfies $DK^y(t) /DK^x(t) \to 0$ as $t \to  0$. 
We already know that in the parabolic case, in general, there is a loss of regularity of the invariant curves at the origin with respect to the regularity of the map  \cite{fontich99}, \cite{BFM20a}, \cite{BFM20b}. Then we can not assume \textit{a priori} a Taylor expansion of high degree of the curve at $t=0$. However, we can obtain formal polynomial approximations, $\MK_n$ and $\mathcal{R}_n$, of $K$ and $R$, satisfying \eqref{semiconjugation} up to a certain order that depends on the degree of differentiability of $F$. Our results will then provide that these expressions are indeed approximations of true invariant  curves, whose existence is rigorously established.  

On the other hand we can suppose that we have approximations, obtained in some 
way, that satisfy some conditions     and obtain that there are true invariant curves closeby.

\subsection{Main results}

First we state the main results concerning the existence of analytic stable invariant manifolds of analytic maps of the form \eqref{forma_normal_cr}.
Since an analytic map of the form \eqref{forma_general} is analytically conjugated to a map of the form \eqref{forma_normal_cr}, the results of the next theorems provide invariant manifolds for  \eqref{forma_general}.
\begin{theorem} \label{teorema_analitic}
	Let $F: U \subset \R^2 \to \R^2$ be an analytic map in a neighborhood $U$ of $(0,0)$ of the form \eqref{forma_normal_cr}. 
	 Assume the following hypotheses according to the different cases:
	 $$
	 \text{ (case 1) }\quad  a_k >0,  \qquad \text{ (case 2) } \quad a_k >0,\, b_l\ne 0, \qquad 
	  \text{ (case 3) } \quad b_l <0.
	 $$
Then, there exists a $C^1$ map  $K: [0, \, \rho ) \to \R^2$, analytic in $(0,\rho)$,  such that  
	\begin{equation} \label{propietats_k_analitic}
	K(t) = \left\{ \begin{array}{ll}
	(t^2,  K^y _{k+1} t^{k+1}) + (O(t^3), O(t^{k+2})) & \text{ case }\; 1, \\
	(t, K^y _{l} t^{l})+ (O(t^2), O(t^{l+1})) & \text{ cases } \; 2, 3 ,
\end{array} \right.
	\end{equation}
with $K^y _{k+1}  = -\sqrt{\frac{2 a_k}{c(k+1)}}$ for case 1,  
$K_l^y = \frac{b_l - \sqrt{b_l^2 + 4 \, c \,a_k\, l}}{2\, c\, l}$ for case 2 and 
$K_l^y=\frac{b_l}{cl}$ for case 3, 
and  a polynomial $R$  of the form $R(t) = t + R_Nt^N + R_{2N-1}t^{2N-1}$, 
with 
$R_k= \frac{c}{2}K^y _{k+1}  $ for case 1 and 
$R_l= c K^y _{l}  $ for cases 2, $3$, such that 
	$$
	F(K(t)) = K(R(t)), \qquad t \in [0, \, \rho).
	$$
\end{theorem}

\begin{remark} 
 This theorem provides a local stable manifold parameterized by $K:[0, \rho) \to \R^2$ with $\rho$ small. The proof does not give an explicit estimate for the value of $\rho$. However, we can extend the domain of $K$ by using the formula 
 $$
 K(t) = F^{-j} K(R^j(t)), \qquad j \geq 1,
 $$
 while the iterates of the inverse map $F^{-1}$ exist (note that $R$ is a weak contraction). In particular, if the map $F^{-1}$ is globally defined, as it happens for example for the H\'enon map, one can extend the domain of $K$ to $[0, \infty)$. This observation also applies for the next theorems \ref{teorema_analitic-a-posteriori}, \ref{teorema_cr} and \ref{teorema_cr-a-posteriori}. In the analytic case the domain of $K$ can be extended to an open domain of $\C$ that contains $(0, \rho)$. 
\end{remark}

Next theorem is an \textit{a posteriori} version of Theorem \ref{teorema_analitic} which, given an analytic approximation, in a certain sense, of the solutions $K$ and $R$ of the conjugation equation $F\circ K = K\circ R$, provides exact solutions of the equation, close to the approximations.

\begin{theorem} \label{teorema_analitic-a-posteriori}
Let $F: U \subset \R^2 \to \R^2$ be as in Theorem \ref{teorema_analitic}
and let $\hat K:(-\rho, \rho) \to \R^2 $ and $\hat R = (-\rho, \rho) \to \R$
be analytic maps satisfying 
$$
\hat K(t) = \left\{ \begin{array}{ll}
(t^2, \hat K^y _{k+1} t^{k+1}) + (O(t^3), O(t^{k+2})) & \quad \text{ case }\; 1 ,\\
(t,\hat K^y _{l} t^{l})+ (O(t^2), O(t^{l+1})) & \quad \text{ cases } \; 2, 3 ,
\end{array} \right.
$$
and $\hat R(t) = t +\hat R_N t^N +O(t^{N+1}) $,  $\hat R_N <0$, such that 
\begin{equation} \label{approx-hyp}
F ( \hat K (t)) -  \hat K (\hat R(t)) =  (O(t^{n+N}), O(t^{n+2N-1})),
\end{equation}
for some $n\ge 2$ in case 1 or $n\ge 1$ in cases 2, 3.

Then, there exists a $C^1$ map $K:[0,\rho) \to \R^2 $, analytic in $(0,\rho)$, and an analytic map $R:(-\rho,\rho) \to \R$  such that 
$$
F ( K (t)) =   K (  R(t)) , \qquad t\in [0,\rho)
$$
and 
$$
 K (t) -  \hat K(t) =  (O(t^{n+1}), O(t^{n+N})) ,
 $$
 $$ 
 R(t) -\hat R(t) = \left\{ \begin{array}{ll}
O(t^{2k-1}) & \text{ if }\; n\le k \\
0 & \text{ if } \; n>k
\end{array} \right. 
\qquad \text{case $1$,}
$$
$$
 R(t) -\hat R(t) = \left\{ \begin{array}{ll}
O(t^{2l-1}) & \text{ if }\; n\le l-1\\
0 & \text{ if } \; n>l-1
\end{array} \right. 
\qquad \text{cases $2,  3$.}
$$
\end{theorem}

\begin{remark} 
In case 1, condition \eqref{approx-hyp} with $n=2$ 
implies the following relations 
$$
 \hat K^y_{k+1} = \pm \sqrt{\frac{2 a_k}{c(k+1)}}, \qquad  \hat R_k = \frac{c}{2} \hat K^y_{k+1}.
$$	
In cases 2 and 3 
 the condition \eqref{approx-hyp} with $n=1$ 
implies 
$$
\hat R_l = c \hat K^y_{k+1}, \qquad  \left\{ \begin{array}{ll}
a_k+b _l \hat K^y_l = l \hat R_l \hat K^y_l & \; \text{ case }\, 2, \\
b_l = l \hat R_l &  \;\text{ case }\,  3.
\end{array} \right. 
$$	
\end{remark}

\begin{remark}
Theorem \ref{teorema_analitic-a-posteriori} provides the existence of a stable manifold assuming it has been previously approximated but the theorem is independent of the way such an approximation has been obtained.  
Propositions \ref{prop_algoritme_1}, \ref{prop_algoritme_23}  and \ref{prop_algoritme_3} (in Section \ref{sec_algoritme}) provide an algorithm to obtain polynomial maps $\MK_{n}$ and $\mathcal{R}_n$ that satisfy condition  \eqref{approx-hyp} of Theorem \ref{teorema_analitic-a-posteriori} for any $n$. 

\end{remark}

\begin{remark}
The form of the map $R$ given in the statement of  Theorem \ref{teorema_analitic} is the normal form of the dynamics of a one-dimensional system in a neighborhood of a parabolic point (see \cite{takens73,voronin81}). 
\end{remark}

The following are the main results concerning the existence and regularity of stable invariant manifolds of $C^r$ maps of the form \eqref{forma_normal_cr}. As in the analytic case, the results provide also the existence of invariant manifolds for maps of the form \eqref{forma_general}.

\begin{theorem} \label{teorema_cr}
	Let $F: U \subset \R^2 \to \R^2$ be a $C^r$ map in a neighborhood $U$ of $(0,0)$ of the form \eqref{forma_normal_cr} with $r \geq 3$.
	
 Assume the following hypotheses according to the different cases:
	\begin{itemize}
		\item (case  $1$)  $a_k >0$ and $r  \geq \frac{3}{2} \, k$,
		\item (case $2$)  $a_k >0$, $b_l\ne 0$, $r > k$ and 
		$$
		\max \Big{\{} \frac{\beta}{(r-2l+2)(r-l+1)}  \big{(} 2l(l-1) + \frac{c \, k\, a_k }{b_l^2} \beta  \big{)} ,  \, \frac{2l \, \beta}{r-l+1} \Big{\}} <1,
		$$
		where $\beta = \frac{2l \, |b_l|}{|b_l - \sqrt{b_l^2 + 4\, c \, a_k \,  l}|}$.
		\item  (case $3$) $b_l <0, \, r > 2l-1$ and  $\frac{l(l-1)}{(r-2l+2)(r-l+1)}<1$.
	\end{itemize}

Then, there exists a $C^1$ map  $H: [0, \, \rho ) \to \R^2$, $H \in C^r(0,  \rho)$, of the form \eqref{propietats_k_analitic}, 
with $H^y _{k+1}  = -\sqrt{\frac{2 a_k}{c(k+1)}}$ for case 1,  
$H_l^y = \frac{b_l - \sqrt{b_l^2 + 4 \, c \,a_k\, l}}{2\, c\, l}$ for case 2 and 
$H_l^y=\frac{b_l}{cl}$ for case 3, 
and  a polynomial $R$  of the form $R(t) = t + R_Nt^N + R_{2N-1}t^{2N-1}$, 
with 
$R_k= \frac{c}{2}H^y _{k+1}  $ for case 1 and 
$R_l= c H^y _{l}  $ for cases 2, $3$, such that 
$$
F(H(t)) = H(R(t)), \qquad t \in [0,  \rho).
$$
	If the map $F$ is  $C^\infty$ then the parameterization $H$ is $C^\infty$ in $(0,\rho)$.


\end{theorem}
\begin{remark}
	The assumptions $a_k >0$ and $k \leq r$ for cases $1$ and $2$ and $b_l<0$ and $l \leq r$ for case $3$  are necessary conditions for the existence of a formal, locally unique stable invariant curve of $F$ asymptotic to $(0,0)$. The other hypotheses of the theorem are nondegeneracy conditions on the reduced form of $F$, sufficient to ensure the existence of a stable invariant curve of class $C^r$ asymptotic to $(0,0)$. We do not claim that these conditions on $r$ are sharp.
\end{remark}

\begin{remark}
	For case $2$, the condition on the coefficients of $F$ is always satisfied provided that $r$ is sufficiently larger than $l$. Another sufficient condition for it to be satisfied is that $\beta$ is small enough. The smallness of the coefficient $\beta$ is a measure of how fast the dynamics on the associated invariant manifold is. 
	For case $3$, a sufficient nondegeneracy condition for the stable manifold to exist is given by $r \geq \frac{4}{3}(2l-1)$. Notice that the assumption $r \geq 2l-1$ is necessary  for the constructions we will do.  
\end{remark}
We also provide an \textit{a posteriori} version of Theorem \ref{teorema_cr}.

\begin{theorem} \label{teorema_cr-a-posteriori}
	Let $F: U \subset \R^2 \to \R^2$ be a map satisfying the hypotheses of Theorem \ref{teorema_cr} and let $\hat K:(-\rho, \rho) \to \R^2 $ and $\hat R  = (-\rho, \rho) \to \R$
	be analytic maps satisfying 
	$$
	\hat K(t) = \left\{ \begin{array}{ll}
	(t^2, \hat K^y _{k+1} t^{k+1}) + (O(t^3), O(t^{k+2})) & \quad \text{ case }\; 1 ,\\
	(t,\hat K^y _{l} t^{l})+ (O(t^2), O(t^{l+1})) & \quad \text{ cases } \; 2, 3 ,
	\end{array} \right.
	$$
	and $\hat R(t) = t +\hat R_N t^N +O(t^{N+1}) $,  $\hat R_N <0$, such that 
	\begin{equation*} 
	F ( \hat K (t)) -  \hat K ( \hat R(t)) =  (O(t^{n+N}), O(t^{n+2N-1})),
	\end{equation*}
	for some $n\ge 2$ in case 1 or $n\ge 1$ in cases 2, 3.
	
	Then, there exists a $C^1$ map $H:[0,\rho) \to \R^2 $, $H \in C^r(0,\rho)$, and an analytic map $R:(-\rho,\rho) \to \R$  such that 
	$$
	F (  H (t)) =  H (  R(t)) , \qquad t\in [0,\rho)
	$$
	and 
	$$
	H (t) -  \hat K(t) =  (O(t^{m}), O(t^{m+N-1})) ,
	$$
	where $m = \min \, \{n+1, \, 2r-2k+2\}$ (case $1$) and $m = \min \, \{n+1, \, r-2l+2\}$ (cases $2, 3$), and  
	$$ 
	R(t) -\hat R(t) = \left\{ \begin{array}{ll}
	O(t^{2k-1}) & \text{ if }\; n\le k \\
	0 & \text{ if } \; n>k
	\end{array} \right. 
	\qquad \text{case $1$,}
	$$
	$$
	R(t) -\hat R(t) = \left\{ \begin{array}{ll}
	O(t^{2l-1}) & \text{ if }\; n\le l-1\\
	0 & \text{ if } \; n>l-1
	\end{array} \right. 
	\qquad \text{cases $2,  3$.}
	$$
\end{theorem}

The structure of the proof is analogous to the one of Theorem \ref{teorema_analitic-a-posteriori} and uses the constructions of the approximations in the proofs of Theorems \ref{teorema_analitic} and \ref{teorema_cr}. It will be omitted. 

As mentioned, using the conjugations $(x,y) \mapsto (\pm x, \pm y)$ and $F^{-1}$ we can obtain the local phase portraits and the location of the local invariant manifolds of $F$ depending on the studied cases (see \cite{fontich99}).

\begin{remark}
	The invariant manifolds obtained in Theorems \ref{teorema_analitic}, \ref{teorema_analitic-a-posteriori}, \ref{teorema_cr} and \ref{teorema_cr-a-posteriori} are unique. For that we refer to Theorem  $4.1$ of \cite{fontich99}, where it is proved that if the map $F$ is $C^k$, in all the considered cases the local stable set $W_{\rho}^{s+}$ is a graph and therefore is unique. This is proved by checking that both the iterates of the points that are above and the ones that are below  the invariant curve cannot converge to the fixed point by a detailed study of the behaviour of the iterates. However, the parameterizations are not unique because if $K$ and and $R$ satisfy $F \circ K = K \circ R$, then for any invertible map $\beta : [0, \rho] \to \R$, the maps $ \tilde{K} = K \circ \beta$ and $\tilde{R} = \beta^{-1} \circ R \circ \beta$ satisfy $F \circ \tilde K = \tilde K \circ \tilde R$. 
\end{remark}

\subsection{Unstable manifolds} \label{subsec_inestable}
Assuming $F$ satisfies the hypotheses of Theorem \ref{teorema_analitic} if $F$ is analytic, or the ones of Theorem \ref{teorema_cr}  if $F$ is  differentiable, in cases $1$ and $2$, the results for the unstable manifolds are obtained from the stated theorems without having to compute the inverse map $F^{-1}$.
Only in case 2 for differentiable maps one has to check a technical condition as explained below.  For case $3$, if one assumes $b_l >0$ instead, then an analogous result is obtained for the existence of an unstable manifold of $F$.

Next, we show that the expansions of the parameterizations of the unstable curves obtained in Section \ref{sec_algoritme}  are approximations of true invariant curves, as it happens for the stable ones.

Assume we have a map of the form \eqref{forma_normal_cr}. Then, by Propositions \ref{prop_algoritme_1}, \ref{prop_algoritme_23} or \ref{prop_algoritme_3} we have approximations $\MK_n$ and $\mathcal{R}_n$ such that
\begin{equation} \label{aprox_inestable}
\MG_n(t) = F(\MK_n(t))- \MK_n(\mathcal{R}_n (t)) = (O(t^{n+N}), O(t^{n+2N-1})),
\end{equation}
with $\mathcal{R}_n(t)= t + R_Nt^N + O(t^{N+1}) $ and $R_N >0$, which means that $0$ is a repellor for $\mathcal{R}_n$. Also, $\mathcal{R}_n$ is invertible and we have
$$
\mathcal{R}_n^{-1}(t) = t - R_N t^N + O(t^{N+1}),
$$
and
$$
F^{-1}\begin{pmatrix} x\\ y \end{pmatrix}
= \begin{pmatrix}
x - cy +  ca_k (x-cy)^k + cb_ly(x-cy)^{l-1} + O(x^{k+1}) + O(yx^{l})\\
y - a_k (x-cy)^k - b_ly(x-cy)^{l-1} + O(x^{k+1}) + O(yx^{l})
\end{pmatrix}.
$$
Then, composing by $F^{-1}$ and $\mathcal{R}_n^{-1}$ in \eqref{aprox_inestable} we obtain
$$
F^{-1} ( \MK_n(t)) - \MK_n ( \mathcal{R}_n^{-1}(t)) =  (O(t^{n+N}), O(t^{n+2N-1})).
$$
Moreover, there exists a change of variables of the form $C(x, y) = (x, -y) + O(\|(x, y)\|^N)$  that transforms $F^{-1}$ into its reduced form   $G := C^{-1} \circ F^{-1} \circ C $, and then $G$ reads
$$
G\begin{pmatrix} x\\ y \end{pmatrix} = \begin{pmatrix}
x + cy   \\
y + a_k x^k -  b_l y x^{l-1}  + O(x^{k+1}) + O(yx^{l})  
\end{pmatrix}.
$$
We also have
$$
G (C^{-1}(\MK_n(t))) - C^{-1}(\MK_n(\mathcal{R}_n^{-1}(t))) = (O(t^{n+N}), O(t^{n+2N-1})).
$$
Thus, if $F$ is in case 1 with $a_k>0$ then $G$ is also in case 1  with the same coefficient $a_k$ positive. Also, if  $F$ is in case 2 with $a_k>0$ and $b_l\ne 0$ then $G$ is also in case 2  with the corresponding coefficients $a_k$ positive and $b_l$ different from 0. If $F$ is in case $3$ with $b_l >0$ then $G$ is also in case $3$ and the coefficient of $y x^{l-1}$ is given by $- b_l$.
 Therefore, by Theorem \ref{teorema_analitic-a-posteriori}
there exist a map $K:[0,\rho) \to \R^2$,  analytic   in $(0, \rho)$ and an analytic  map $R: (-\rho, \rho)\to \R$  such that $G \circ K = K \circ R$, with
\begin{equation} \label{apr_ca_inestable}
K (t) -  C^{-1} \MK_n(t) =  (O(t^{n+1}), O(t^{n+N})) ,
\end{equation}
$$
R(t) - \mathcal{R}_n^{-1}(t) = \left\{ \begin{array}{ll}
O(t^{2k-1}) & \text{ if }\; n\le k \\
0 & \text{ if } \; n>k
\end{array} \right.
\qquad \text{case $1$,}
$$
$$
R(t) - \mathcal{R}_n^{-1}(t) = \left\{ \begin{array}{ll}
O(t^{2l-1}) & \text{ if }\; n\le l-1\\
0 & \text{ if } \; n>l-1
\end{array} \right.
\qquad \text{cases $2, 3$.}
$$
Hence, we have $F^{-1} \circ C \circ K = C \circ K \circ R$, which means that $C \circ K$ is a parameterization of an unstable manifold of $F$. Moreover, from  \eqref{apr_ca_inestable} and the form of $C$, we have
$$
C(K(t))   - \MK_n(t) = (O(t^{n+1}), O(t^{n+N})),
$$
and therefore $\MK_n$ is an approximation of a parameterization of such unstable manifold.

In the $C^r$ case one has to apply
Theorem \ref{teorema_cr-a-posteriori}. If $F$ satisfies the conditions of case 1, $G$ also does. The same happens for case 3 if we assume $b_l >0$ instead of $b_l <0$. If  $F$ satisfies the conditions of case 2, since the coefficient $b_l$ of $F$ becomes $-b_l$ for $G$, one has to check the condition involving the maximum taking now $\beta $ as
$\beta = \frac{2l \, |b_l|}{|-b_l - \sqrt{b_l^2 + 4\, c \, a_k \,  l}|}$. Then, for cases 1 and 3 or for case 2 when that condition holds, we conclude as we have explained for the analytic case.

\section{Formal polynomial approximation of the parameterizations of the  cur\-ves} \label{sec_algoritme}

 In this section we consider $C^r$ maps 
  $F$  of the form \eqref{forma_normal_cr} and we provide algorithms, depending on the case, to obtain two polynomial maps, $\mathcal{K}_n$ and $\mathcal{R}_n$, that  are approximations of solutions $K$ and $R$ of the invariance equation 
\begin{equation} \label{s3:conj-eq}
 F \circ K = K \circ R.
\end{equation}
Because of the nature of the problem, the two components of $\mathcal{K}_n$ will have a different order and different degrees. The index $n$ has  to be seen as an induction index. Higher values of $n$ mean better approximation. 

The obtained approximations correspond to formal invariant curves. They correspond to stable curves when the coefficient $R_k$ (case 1) or $R_l$ (cases 2, 3) of $\mathcal{R}_n$ are negative (see below). When those coefficients are positive they correspond to unstable curves.

\begin{proposition}[Case $1$] \label{prop_algoritme_1}
	Let $F$ be a $C^r$ map of the form \eqref{forma_normal_cr} with  $2\le k \leq r$. Assume that $k < 2l-1$ and $a_k > 0$. Then, for all $2 \leq n \leq 2(r-k+1)$, there exist two pairs of polynomial maps,  $\MK_n$ and $\mathcal{R}_n$, of the form
	$$ 
	\MK_{n} (t)= 
	\begin{pmatrix}
	t^2 + \cdots + K_n^x  t^n \\
	K_{k+1}^y t^{k+1} + \cdots +  K_{n+k-1}^y t^{n+k-1} 
	\end{pmatrix}
	$$
	and 
	$$ \mathcal{R}_n(t) =
	\begin{cases}
	t+R_kt^k &\qquad \text{ if }  \;\; 2 \leq n \leq k, \\
	t+R_kt^k + R_{2k-1}t^{2k-1} & \qquad \text{ if } \;\; n\ge k+1, 
	\end{cases}
	$$
	such that 
	\begin{equation} \label{eq_prop1}
	\MG_n (t) := F(\MK_{n}(t))- \MK_{n}(\mathcal{R}_n(t)) = (O(t^{n+k}),\, O(t^{n+2k-1})).
	\end{equation}
	For the first pair we have
		$$
	K_{k+1}^y = -  \sqrt{\frac{2 \, a_k}{c\, (k+1)}}, \qquad R_k = -\sqrt{\frac{c\, a_k}{2(k+1)}}= \frac{c}{2} K_{k+1}^y ,
	$$
	and for the second one
	$$
	K_{k+1}^y =  \sqrt{\frac{2 \, a_k}{c\, (k+1)}}, \qquad R_k = \sqrt{\frac{c\, a_k}{2(k+1)}} = \frac{c}{2} K_{k+1}^y .
	$$

If $F$ is $C^\infty$ or analytic, one can compute the polynomial approximation $\MK_n$ up to any order.
\end{proposition}

\begin{remark}
The algorithm described in the proof of this (and the next) propositions can be implemented in a computer program to calculate $\mathcal{R}$ and the expansion of  $\MK_n$.
\end{remark}

\begin{notation} Along the proof, given a  $C^r$ one-variable map $f$, we will denote $[f]_n$, $0 \leq n \leq r$, the coefficient of the term of order $n$ of the jet of $f$ at $0$. 
\end{notation}

\begin{proof} We will see that we can determine $\MK_n$ and $\mathcal{R}_n$ iteratively. 
	
For $n=2$, we claim that there exist polynomial maps $\MK_{2}(t) = (t^2,  K_{k+1}^y  t^{k+1})$ and $\mathcal{R}_2(t) = t + R_k  t^k$, such that $\MG_2 (t) = F(\MK_{2}(t))- \MK_{2}(\mathcal{R}_2(t)) = (O(t^{k+2}),  O(t^{2k+1}))$.

Indeed, from the expansion of $\MG_2$ we have
	$$ \MG_{2}(t) = 
	\begin{pmatrix}
	t^2 + c \, K_{k+1}^y t^{k+1}  - t^2 - 2R_kt^{k+1} + O(t^{2k})  \\
	K_{k+1}^y t^{k+1} + a_kt^{2k} - K_{k+1}^y t^{k+1} - (k+1) K_{k+1}^yR_kt^{2k} + 
	O(t^{2k+1})
	\end{pmatrix},
	$$ 
	so, if the conditions
	\begin{align*}
	& c \, K_{k+1}^y - 2R_k= 0,  \qquad a_k -(k+1) K_{k+1}^yR_k = 0,
	\end{align*}
	are satisfied, then we clearly have $\MG_2 (t) =  (O(t^{2+k}),\, O(t^{2k+1}))$, and we obtain the values of $K_{k+1}^y$ and $R_k$ given in the statement. 

	Now we assume that we have already obtained maps $\MK_n$ and $\mathcal{R}_n$, $2\le n < 2(r-k+1)$ such that \eqref{eq_prop1} holds true, and we look for 
	$$
	\MK_{n+1} (t) = \MK_n(t) + \begin{pmatrix} K_{n+1}^x \, t^{n+1} \\ K_{n+k}^y \, t^{n+k} \end{pmatrix},  \qquad \mathcal{R}_{n+1}(t) = \mathcal{R}_n(t) + R_{n+k-1} \, t^{n+k-1},
	$$
such that  $\MG_{n+1}(t) = (O(t^{n+k+1}), \, O(t^{n+2k}))$.

Using Taylor's theorem, we write
	\begin{align*}
	& \MG_{n+1}(t)  = F(\MK_n(t) + (K_{n+1}^x \, t^{n+1}, \, K_{n+k}^y\, t^{n+k}) ) \\
	& \quad - (\MK_n(t) + (K_{n+1}^x \, t^{n+1}, \, K_{n+k}^y\, t^{n+k}))\circ (\mathcal{R}_n(t) + R_{n+k-1}\, t^{n+k-1})  \\
	& = \MG_n(t) + DF(K_n(t)) \cdot (K_{n+1}^x  t^{n+1},  K_{n+k}^y t^{n+k}) \\
	& \quad - (K_{n+1}^x  t^{n+1},  K_{n+k}^y\, t^{n+k})\circ (\mathcal{R}_n(t) + R_{n+k-1}\, t^{n+k-1}) \\
	& \quad+ \int_0^1 \hspace{-1pt} (1-s)  D^2F(\MK_n(t) + s(K_{n+1}^x  t^{n+1}, \, K_{n+k}^y t^{n+k})) \, ds \,  (K_{n+1}^x  t^{n+1}, \, K_{n+k}^y t^{n+k})^{\otimes 2} \\
	& \quad - D\MK_n(\mathcal{R}_n(t)) R_{n+k-1} \, t^{n+k-1}   \\
	& \quad - \int_0^1 \hspace{-1pt} (1-s)  D^2\MK_n (\mathcal{R}_n(t) + s  \, R_{n+k-1} \, t^{n+k-1}) \, ds \, (R_{n+k-1} \, t^{n+k-1})^2.
	\end{align*}
	
	Performing the computations in the previous expression we have
	\begin{align} 
	 \begin{split} \label{expansio_gn}
	& \MG_{n+1}(t) = \MG_{n}(t)   \\
	 & + \begin{pmatrix}
	 [c\, K_{n+k}^y - (n+1)R_k \, K_{n+1}^x  - 2\, R_{n+k-1}] \, t^{n+k} + O(t^{n+k+1}) \\
	   [k\, a_k \, K_{n+1}^x - (n+k) \, R_k \, K_{n+k}^y  -(k+1) K_{k+1}^y \, R_{n+k-1}] \, t^{n+2k-1} + O(t^{n+2k})  \end{pmatrix}.
	\end{split}
	\end{align}
	Since, by the induction hypothesis,
	  $\MG_{n}(t) = (O(t^{n+k}), \, O(t^{n+2k-1}))$,  to complete the induction step we need to make $[\MG_{n+1}^x]_{n+k}$ and $[\MG_{n+1}^y]_{n+2k-1}$ vanish.
	
	From \eqref{expansio_gn} we have
	\begin{align*}
& 	[\MG_{n+1}^x]_{n+k} = [\MG_{n}^x]_{n+k}+ c\, K_{n+k}^y - (n+1)R_k \, K_{n+1}^x  - 2\, R_{n+k-1},  \\
& [\MG_{n+1}^y]_{n+2k-1} = [\MG_n^y]_{n+2k-1} + k\, a_k \, K_{n+1}^x - (n+k) \, R_k \, K_{n+k}^y  -(k+1) K_{k+1}^y \, R_{n+k-1} .
	\end{align*}
	
	Thus, the conditions $	[\MG_{n+1}^x]_{n+k}   = [\MG_{n+1}^y]_{n+2k-1} =0$ are equivalent to
		\begin{equation} \label{sist_lineal}
	\begin{pmatrix}
	-(n+1)R_k & c \\
	k \, a_k & -(n+k) \, R_k
	\end{pmatrix} \hspace{-2pt}
	\begin{pmatrix}
	K_{n+1}^x \\
	K_{n+k}^y
	\end{pmatrix} = 
	\begin{pmatrix}
	- [\MG_n^x]_{n+k} + 2 \, R_{n+k-1} \\
	- [\MG_n^y]_{n+2k-1}  + (k+1) \, K_{k+1}^y \, R_{n+k-1}
	\end{pmatrix}.
	\end{equation}

	If $n \neq k$ the matrix in the left hand side of \eqref{sist_lineal} is invertible, so we can take $R_{n+k-1} = 0$ and  then obtain $K_{n+1}^x$ and $K_{n+k}^y$ in a unique way. When $n=k$, the determinant of the matrix is zero. Then, choosing 
	$$
	R_{2k-1} =  \frac{2k \, R_k \, [\MG_n^x]_{2k} + c \, [\MG_n^y]_{3k-2}}{2 \, (3k+1) \, R_k},
	$$
	system \eqref{sist_lineal} has solutions. In this case, however, $K_{k+1}^x$ and $K_{2k}^y$ are not uniquely determined. 
\end{proof}

\begin{proposition}[Case $2$] \label{prop_algoritme_23}
Let $F$ be a $C^r$ map of the form \eqref{forma_normal_cr}, with $r \ge k\ge  2$.
We assume $k=2l-1$, $a_k \neq 0$, $b_l \neq 0$ and $a_k > - \frac{b_l^2}{4cl}$. If $a_k <0$ we  assume also  $a_k \neq -\frac{2l+1}{3l-1}\, b_l^2$. 
Then,   for all $1 \leq n \leq r-2l+2=r-k+1$, there exists two pairs of polynomial functions $\MK_n$ and $\mathcal{R}_n$ of the form 
\begin{equation} \label{forma_Kn2}
\MK_{n} (t)= 
\begin{pmatrix}
t + \cdots + K_n^x  t^n \\
K_{l}^y t^{l} + \cdots +  K_{n+l-1}^y t^{n+l-1} 
\end{pmatrix}
\end{equation}
and
\begin{equation} \label{forma_Rn2}
\mathcal{R}_n(t) =
\begin{cases}
t+R_l t^l &\quad  \text{ if }  \ 1 \leq n \leq l-1, \\
t+R_l t^l + R_{2l-1}t^{2l-1} & \quad \text{ if } \ n \ge l, \\
\end{cases}
\end{equation}
such that 
\begin{equation*} \label{eq_prop2}
\MG_n (t) := F(\MK_{n}(t))- \MK_{n}(\mathcal{R}_n(t)) = (O(t^{n+l}),\, O(t^{n+2l-1})).
\end{equation*}

For the first pair we have 
$$
K_l^y = \frac{b_l - \sqrt{b_l^2 + 4 \, c \,a_k\, l}}{2\, c\, l}, \qquad  R_l = \frac{b_l - \sqrt{b_l^2 + 4 \, c \,a_k\, l}}{2l} = cK_l^y ,
$$
and for the second one
$$
K_l^y = \frac{b_l + \sqrt{b_l^2 + 4 \, c \,a_k\, l}}{2\, c\, l}, \qquad  R_l = \frac{b_l + \sqrt{b_l^2 + 4 \, c \,a_k\, l}}{2l} = cK_l^y.
$$ 

If  $a_k = -\frac{2l+1}{3l-1}\, b_l^2$ and $b_l<0$ we can compute the first pair up to $n=l-1$ and the second pair for any $n\le r-2l+2$.
If   $a_k = -\frac{2l+1}{3l-1}\, b_l^2$ and $b_l>0$ we can compute the first pair up to $n\le r-2l+2$  and the second pair up to $n=l-1$.

If $F$ is $C^\infty$ or analytic, one can compute the polynomial approximations $\MK_n$ up to any order, except when $a_k = -\frac{2l+1}{3l-1}\, b_l^2$.
\end{proposition}

 \begin{proposition}[Case $3$] \label{prop_algoritme_3}
	Let $F$ be a $C^r$ map of the form \eqref{forma_normal_cr}, with $r \ge l\ge  2$.
	Assume $k>2l-1$,  $b_l \neq 0$  
	Then,   for all $1 \leq n \leq r-2l+2$, there exist a pair of polynomial functions $\MK_n$ and $\mathcal{R}_n$ of the form \eqref{forma_Kn2} and \eqref{forma_Rn2} respectively,
	such that 
	\begin{equation*} \label{eq_prop2}
	\MG_n (t) := F(\MK_{n}(t))- \MK_{n}(\mathcal{R}_n(t)) = (O(t^{n+l}),\, O(t^{n+2l-1})).
	\end{equation*}
	We have 
	$$
	K_l^y= \frac{b_l}{c\, l} , \qquad  R_l =\frac{b_l}{l}= cK_l^y.
	$$

	If we further assume that $k\le r$ and $a_k\ne 0$, then for $1\le n\le r-(k-l)l-2l+1$ there exists another pair  $\MK_n$ and $\mathcal{R}_n$ with 
	\begin{equation*} \label{forma_Kncase32 }
	\MK_{n} (t)= 
	\begin{pmatrix}
	t + \cdots + K_n^x  t^n \\
	K_{k-l+1}^y t^{k-l+1} + \cdots +  K_{n+k-l}^y t^{n+k-l} 
	\end{pmatrix}
	\end{equation*}
	and
	\begin{equation*} \label{forma_Rncase32}
	\mathcal{R}_n(t) =
	\begin{cases}
	t+R_{k-l+1} t^{k-l+1} &\quad  \text{ if }  \ 2 \leq n \leq k-l, \\
	t+R_{k-l+1} t^{k-l+1} + R_{2(k-l)+1}t^{2(k-l)+1} & \quad \text{ if } \ n \ge k-l+1, \\
	\end{cases}
	\end{equation*}
	such that 
	\begin{equation*} \label{eq_prop2case32}
	\MG_n (t) := F(\MK_{n}(t))- \MK_{n}(\mathcal{R}_n(t)) = (O(t^{n+k-l+1}),\, O(t^{n+k})).
	\end{equation*}
	We have 
		$$
	K_{k-l+1}^y= -\frac{a_k}{b_l} , \qquad  R_{k-l+1} = cK_{k-l+1}^y.
	$$
	
	If $F$ is $C^\infty$ or analytic, one can compute the polynomial approximations $\MK_n$ up to any order.
\end{proposition}

The proofs are analogous to the one of Proposition \ref{prop_algoritme_1}.

\section{The analytic case} \label{sec_analytic}

This section is devoted to prove Theorems \ref{teorema_analitic}
	and \ref{teorema_analitic-a-posteriori}. 
Following the parameterization method, given a map $F$ of the form \eqref{forma_normal_cr}, first we consider  polynomial approximations  $\MK_n : \R \to \R^2$ and $\mathcal{R}_n: \R \to \R$   of solutions of equation \eqref{s3:conj-eq} obtained in Section \ref{sec_algoritme} up to a high enough order, to be determined in the proof. Then, keeping $R=\mathcal{R}_n$ fixed, we look for a correction $\Delta: [0, \, \rho) \to \R^2$, for some $\rho>0$, of $\MK_n$, analytic on $(0,\rho)$, such that the pair $K= \MK_n + \Delta $, $R=\mathcal{R}_n$  satisfies the invariance condition
\begin{equation} \label{eq_delta_analitic_1}
F \circ (\MK_n + \Delta) - (\MK_n + \Delta) \circ R = 0.
\end{equation}
 
The proof of Theorem \ref{teorema_analitic} is organized as follows. First, taking into account the structure of $F$  we rewrite equation \eqref{eq_delta_analitic_1}  to separate the dominant linear part with respect to $\Delta$ and the remaining terms. This motivates the introduction of two families of operators, $\MS_{n, \, R}$ and $\MN_{n, \, F}$, and the spaces where these operators will act on. We provide the properties of these operators in Lemmas \ref{invers_S}  and \ref{lema_N_analitic}.

Finally, we rewrite the equation for $\Delta $ as the fixed point equation
$$
\Delta = \MT_{n,\, F} (\Delta), \qquad \text{where}\qquad \MT_{n, \, F} = \MS_{n, \, R}^{-1} \circ \MN_{n, \, F}
$$
and we apply the Banach fixed point theorem to get the solution. The properties of the operators $\MT_{n, F}$ are deduced in  Lemma \ref{lema_contraccio_analitic}. At the end of the section we prove Theorem \ref{teorema_analitic-a-posteriori}.

\subsection{The functional equation} \label{sec_eqfuncional_analitic}

Let $F: U \subset \R^2 \to \R^2$ be an analytic map in a neighborhood $U$ of $(0,0)$, satisfying the hypotheses of Theorem \ref{teorema_analitic},

\begin{equation*}
F(x, y) = 
\begin{pmatrix}
x + c\, y \\
\quad \ \,  y
\end{pmatrix}+
\begin{pmatrix}
0 \\
p(x) + y \, q(x) +  u(x, \, y) + g(x, y) 
\end{pmatrix},
\end{equation*}
where $c>0$, $p$, $q$ and $u$ are the polynomials introduced in Section \ref{sec:notation} and  $g(x, y)$ is an analytic function. We take $p$, $q$ and $u$ of degree at least $k$ in case $1$ and degree at least $2l-1$ in cases $2$ and $3$. Then we have $g(x,y) = O(\|(x,y)\|^{k+1})$ for case $1$ and $g(x,y) = O(\|(x,y)\|^{2l})$ for cases $2$ and $3$. We denote $v(x, y) = u (x,y) + g(x, y)$.

From Propositions \ref{prop_algoritme_1}, \ref{prop_algoritme_23}  and \ref{prop_algoritme_3} we take $n$, with $n \geq k+1 $ in case $1$ and $n \geq  l$ is cases $2$ and $3$, and we have that there exist  polynomials $\MK_n$ and $R= \mathcal{R}_n$ such that 
\begin{equation} \label{analitic_hipotesi}
\ME_n(t) = (O(t^{n+N}), \,O(t^{n+2N-1}) ) ,
\end{equation}
where $\ME_n = F \circ \MK_n - \MK_n \circ R$.
Since we are looking for the stable manifold we will take the approximations corresponding to $R=\mathcal{R}_n$ with the coefficient $R_N <0$.

Hence, we look for $\rho>0$ and a map $K= \MK_n + \Delta: [0, \, \rho) \to \R^2$, analytic on $(0,\rho)$ satisfying \eqref{eq_delta_analitic_1},
where $\MK_n$ and $R$ are the mentioned maps that satisfy \eqref{analitic_hipotesi}. Moreover, we will ask $\Delta$ to  satisfy $\Delta = (\Delta^x, \Delta^y) = (O(t^n), \, O(t^{n+N-1}))$.

Using \eqref{analitic_hipotesi} we  can rewrite \eqref{eq_delta_analitic_1} as 
\begin{align}  \begin{split}\label{eqdelta_analitic}
\Delta^x \circ R - \Delta^x  &= c \,  \Delta^y + \ME_n^x ,  \\
\Delta^y \circ R -  \Delta^y  & = p\circ(\MK_n^x + \Delta^x) - p\circ \MK_n^x  + \MK_n^y \cdot (q\circ(\MK_n^x + \Delta^x) - q\circ \MK_n^x) \\
& \quad+ \,  \Delta^y \cdot q\circ (\MK_n^x + \Delta^x) + v\circ (\MK_n + \Delta) - v \circ \MK_n + \ME_n^y.
\end{split}
\end{align}

\subsection{Function spaces, the operators $\MS_{n, \, N}$ and $\MN_{n, \, N}$ and their properties} \label{sec_operadors_analitic}

Next we introduce notation, suitable function spaces, and some operators. 

\begin{definition}
Given $\beta, \rho >0$ such that $ \rho <1$ and $ \beta < \pi$,  let $S$ be the sector
$$
S = S(\beta, \rho) = \big{\{} z \in \C \, \, | \, \, |\arg(z)| < \frac{\beta}{2}, \, 0 < |z| < \rho \big{\}}.
$$
Given a sector $S = S(\beta, \, \rho)$ let $\MX_n$,  for $n \in \N$, be the Banach space given by
\begin{align*}
\mathcal{X}_{n} = \, & \{ f : S \rightarrow \C \ | \ f \in \textrm{Hol}  (S) , \;  f((0,  \rho)) \subset \R,  \; \|f\|_n: = \sup_{z \in S } \frac{|f(z)|}{|z|^n} < \infty \},  
\end{align*}
where $\text{Hol}(S)$ denotes the space of holomorphic functions on $S$.
\end{definition}
Note that when $n\ge 1$ the functions $f$ in $\MX_n$ can be continuously extended to $z=0$ with $f(0)=0$ and, if moreover, $n\ge 2$, the derivative of $f$ can be continuously extended  to  $z=0$ with $f'(0)=0$.

Note also that $\mathcal{X}_{n+1} \subset \mathcal{X}_n$, for all $n \in  \mathbb{N}$,
and that if  $f\in \mathcal{X}_{n+1} $, then $\|f\|_n \le \|f\|_{n+1} $. Moreover if $f\in \MX_m, \, g \in  \MX_n$, then $f g \in \MX_{m+n}$ and $\|fg\|_{m+n} \leq \|f\|_m \,\|g\|_n.$

Given $n, \, m \in \N$ we denote $\mathcal{X}_{m, \, n}: = \mathcal{X}_m \times \mathcal{X}_{n}$ the  product spaces, endowed with the product norm
$$
\|f\|_{m, \, n} = \max \, \{ \|f^x\|_m, \,   \|f^y\|_n \}, \qquad f=(f^x, f^y)\in \mathcal{X}_{m, \, n}.
$$ 

Given $n\ge 1, \, N\ge 2$, we define the space
\begin{equation*}
\Sigma_{n, \, N} =  \MX_{n,  \, n+N-1}, 
\end{equation*}
endowed with the product norm. Also, given $\alpha >0$, we define the closed ball
\begin{align*}
\Sigma_{n, \, N}^\alpha &= \{f \in \Sigma_{n, \, N} \ | \ \|f\|_{\Sigma_{n, \, N}} \leq \alpha\}.
\end{align*}
For the sake of simplicity, we will omit the parameters $\rho$ and $\beta$ in the notation of the spaces $\Sigma_{n, \, N}$ and the balls $\Sigma^\alpha _{n, \, N}$.

Now let $F$ be as in Theorem \ref{teorema_analitic}, and  $\MK_n$ and $R= \mathcal{R}_n$ be  the polynomials  provided in Section \ref{sec_algoritme} satisfying \eqref{analitic_hipotesi} with $n\ge k+1$ in case 1 and $n\ge l$ in cases 2, 3.

Since $F$ is analytic in $U$, it has a holomorphic extension to some neighborhood $W$ of $(0,0)$ in $\C^2$. Let $d>0$ be the radius of a ball in $\mathbb{C}^2$ contained in the domain where $F$ is holomorphic.
Also, $\MK_n$ and $R$ are defined on any complex sector $S(\beta, \, \rho)$.  
Then it is possible to set equation \eqref{eqdelta_analitic} in a space of holomorphic functions defined in a sector $S(\beta, \, \rho)$, and look for $\Delta$ being an analytic function of a complex variable that takes real values when restricted to the real line. 

To solve equation \eqref{eqdelta_analitic}, we will consider $n $ big enough and we will look for a solution, $\Delta$,  in a closed ball of the space $\Sigma_{n, \, N}$. 
In order for the compositions in \eqref{eqdelta_analitic} to make sense
 we need to ensure the range of $\MK_n+ \Delta $ to be contained in the domain where $F$ is analytic. We take 
$$
\alpha = \min \,  \big{\{}  \tfrac{1}{2}, \, \tfrac{d}{2} \big{\}}.
$$ 
In this way, since $\MK_n(0)=(0,0)$, taking 
$\rho_K \in (0, \, 1)$ such that 
$  \sup_{z \in S (\beta, \, \rho_K)}  \|\MK_n(z)\| \\ < d/2
$ 
and $\rho\le \rho_K$, if $\Delta :  S (\beta, \, \rho) \to \C^2$ belongs to the ball of radius $\alpha $ of $\MX_{n, \, m}$, with
 $n, \, m \geq 0$, we have
	\begin{align*}
\sup_{z \in S (\beta, \, \rho)} \|\Delta (z)\| 
&= \sup_{z \in S (\beta, \, \rho)} \, \max \{\, |\Delta^x(z)|,  |\Delta^y(z)| \} 
 \leq \, \max \, \{\tfrac{d}{2} \, \rho^{n}, \, \tfrac{d}{2} \, \rho^{m} \} \,  < \frac{d}{2}.
\end{align*}  

Therefore, under the previous conditions, if $\rho\le \rho_K$ and $\Delta \in 
 \Sigma^\alpha _{n, \, N}$ then $\|\MK_n(z) +\Delta(z)\| < d$
and the composition $F\circ (\MK_n + \Delta)$ is well defined.

Next we introduce two families of operators that will be used to deal with  \eqref{eqdelta_analitic}. The definition of such operators is motivated by the equation itself. 

First, we state the following auxiliary result (see \cite{BFM17}), 
\begin{lemma} \label{lema_sector}
	Let $R: S(\beta, \rho) \rightarrow \C$ be a holomorphic function of the form $R(z) = z +R_N z^N + O(|z|^{N+1})$, with $R_N<0$. Assume that $0<\beta < \frac{\pi}{N-1}$. Then, for any $\nu \in (0,\,   (N-1)|R_N| \cos \lambda)$, with $\lambda = \beta \,  \frac{N-1}{2}$, there exists $\rho>0$ small enough such that 
	$$
	|R^j(z)| \leq \frac{|z|}{(1+ j \, \nu \, |z|^{N-1})^{1/N-1}}, \qquad \forall \, j \in \mathbb{N}, \quad \forall \, z \in S(\beta, \, \rho),
	$$
	where $R^j$ refers to the $j$-th iterate of the map $R$.
	In addition, $R$ maps $S(\beta, \rho)$ into itself.
\end{lemma}

Then, if  $f$ is defined in $S(\beta, \, \rho)$, with suitable values of the parameters $\beta, \, \rho$, and $R$ satisfies the conditions of the  lemma, the composition $f \circ R$ is well defined. 

\begin{definition} \label{def_S_analitic}
	Given $n \geq 1, \, N\geq 2$ and a polynomial $R(z) = z +R_N z^N + O(|z|^{N+1})$ satisfying the hypotheses of Lemma \ref{lema_sector}, let 
	$\MS_{n, \, R}:  \Sigma_{n, \, N}\rightarrow  \Sigma_{n, \, N}$ be the linear operator defined  component-wise as $\MS_{n, \, R} = (\MS^x_{n, \, R}, \, \MS^y_{n, \, R})$, with
	\begin{align*}
	\MS^x_{n, \, R}\, f =  \MS^y_{n, \, R} \, f &=  f \circ R  - f.
	\end{align*} 
\end{definition}

\begin{remark}
Notice that although both components of $\MS_{n, \, R}$ are formally identical they act on spaces of holomorphic functions of different orders.
\end{remark}

\begin{definition} \label{def_N_analitic}
	Let $F$ be the holomorphic extension of an analytic map of the form \eqref{forma_normal_cr} satisfying the hypotheses of Theorem \ref{teorema_analitic}.
	For $n\in \N$, we introduce $\MN_{n, \, F} = (\MN^x_{n, \, F},\MN^y_{n, \, F}): \Sigma_{n, \,  N}^\alpha \to \MX_{n+N-1, \, n+2N-2}$, by 
	\begin{align*}
	\MN_{n, \, F}^x (f) &= c \, f^y  + \ME_n^x, \\
	\MN_{n, \, F}^y(f) & =  p\circ (\MK_n^x+ f^x) - p \circ \MK_n^x + \MK_n^y \cdot (q \circ (\MK_n^x+f^x)-q \circ \MK_n^x)\\
	& \quad + f^y \cdot q \circ (\MK_n^x + f^x) + v \circ (\MK_n+ f) - v \circ \MK_n+ \ME_n^y.
	\end{align*}
\end{definition}

By the properties of $R$  and the choice of $\alpha$, the operators $\MS_{n, \, R}$ and $\MN_{n, \, F}$ are well defined and $\MS_{n, \, R}$ is linear and bounded.

Using  these operators, equations \eqref{eqdelta_analitic} can be written as 
\begin{equation*} 
\MS_{n, \, R}\, \Delta = \MN_{n, \, F} (\Delta ).
\end{equation*}

The following lemma states  that  the operators $\MS_{n,\,  R}$ have a bounded right inverse and provide a bound for the norm $\|\MS_{n, \, R}^{-1}\|$.  

\begin{lemma} \label{invers_S}
		Given $N\geq 2$ and $n \geq 1$, the operator $\MS_{n, \, R} : \Sigma_{n, \, N} \to \Sigma_{n, \, N} $, has a bounded right inverse 
		$$
		\MS_{n, \, R}^{-1} : \MX_{n+N-1, \, n+2N-2} \to  \Sigma_{n, \, N}
		= \MX_{n, \, n+N-1},
		$$
		given by 
		\begin{equation} \label{sol_parabolic}
		\MS_{n, \, R}^{  -1} \, \eta = - \, \sum_{j=0}^\infty \, \eta \circ R^j,  \qquad \eta \in \MX_{n+N-1, \, n+2N-2}.
		\end{equation}
		Moreover, for any fixed $\nu \in (0, \,\,  (N-1) |R_N |)$, there exists $\rho >0$  such that, taking $S(\beta, \, \rho)$ with $\beta < \tfrac{\pi}{N-1}$ as the domain of the functions of $\MX_{n+N-1, \,  n+2N-2}$,  we have the operator norm bounds
		$$
		\|(\MS^x_{n, \, R})^{-1}\| \leq \rho^{N-1}  + \tfrac{1}{\nu} \, \tfrac{N-1}{n}, \qquad
		\|(\MS^y_{n, \, R})^{-1}\| \leq \rho^{N-1}  + \tfrac{1}{\nu} \, \tfrac{N-1}{n+N-1}.
		$$
\end{lemma}

The operators $\MN_{n, \, F}$ are Lipschitz  and we provide  bounds for their Lipschitz constants. 

\begin{lemma} \label{lema_N_analitic} For each $n \geq 3$, there exists  a constant, $M_n >0$, for which  the operator $\MN_{n, \, F}$  satisfies
	$$
	\text{\emph{Lip}} \ \MN^x_{n, \, F} = c,
	$$
	and 
	\begin{align*}
	\text{\emph{Lip}} \ \MN^y_{n, \, F} & \leq   k\, |a_k| +  M_n  \rho, \quad \text{(case $1$)}, \\
	\text{\emph{Lip}} \ \MN^y_{n, \, F} &  \leq  \max  \{ ( (l-1) \, |K_l^y  \, b_l| + k \, |a_k|) + M_n  \rho, \,  |b_l| + M_n  \rho \}, \quad \text{(case $2$)}, \\
	\text{\emph{Lip}} \ \MN^y_{n, \, F} & \leq  \max  \{(l-1) \, |K_l^y  \, b_l| + M_n \rho, \, |b_l| + M_n  \rho \}, \quad \text{(case $3$)},
	\end{align*}
	where $\rho$ is the radius of the sector $S(\beta, \, \rho)$ where the functions of $\Sigma_{n, \, N}^\alpha$ are defined. 
\end{lemma}

Now, we define  the third family of operators, $\MT_{n, \, F}$.

\begin{definition} \label{def_T_analitic} 	Let $F$ be the holomorphic extension of an analytic map of the form \eqref{forma_normal_cr} satisfying the hypotheses of Theorem \ref{teorema_analitic}.
Given $n\ge 3$ we define $\MT_{n, \, F} : \Sigma_{n, \,  N}^\alpha \to \Sigma_{n, \,  N}$ by 
	$$
	\MT_{n, \, F} = \MS_{n, \, R}^{-1} \circ \MN_{n, \, F}.
	$$
\end{definition}

\begin{remark}
	Note that given a map $F$, to define the previous operators we always take together the associated triple $(F, \, \mathcal{K}_n, \, R)$ satisfying $F \circ \MK_n - \MK_n \circ  R = \mathcal{E}_n$. Then, the operators $\MS_{n, \, R}, \,  \MN_{n, \, F}$ and $\MT_{n, \, F}$ are associated not only with the map $F$ itself but to the approximation of a particular invariant manifold of $F$. 
\end{remark}

\begin{lemma} \label{lema_contraccio_analitic}
	Given an analytic map $F$ satisfying the hypotheses of Theorem \ref{teorema_analitic}, there exist $n_0 > 0$ and $\rho_0 >0$ such that if $\rho < \rho_0$, then, for every $n \geq n_0$, we have $\MT_{n, \, F} (\Sigma_{n, \, N}^\alpha) \subseteq \Sigma_{n, \, N}^\alpha$ and $\MT_{n , \, F}$ is a contraction operator in that ball. 
\end{lemma}
The proofs of the previous three lemmas are deferred to Section \ref{sec_dem_lemes}.

\subsection{Proofs of Theorems \ref{teorema_analitic} and \ref{teorema_analitic-a-posteriori}} \label{sec_dem_teorema_analitic}

Now we are ready to give the proofs of Theorems \ref{teorema_analitic} 
and \ref{teorema_analitic-a-posteriori}.

\begin{proof}[Proof of Theorem \ref{teorema_analitic}]
First we consider the holomorphic extension of $F$ to a neighborhood of the origin which contains a ball of radius $d>0$ in $\C^2$ and let
$\alpha= \min \, \{1/2, d/2\}$. Let  $\MK_n$ and $R(t) = \mathcal{R}_n(t)  = t + R_N t^N + R_{2N-1}t^{2N-1}$ be the polynomials  given by Propositions \ref{prop_algoritme_1},  \ref{prop_algoritme_23} or \ref{prop_algoritme_3} , with $n\ge k+1$ or $n\ge l$ respectively, satisfying
\begin{equation*} \label{dem_analitic_polinomi}
\ME_n(t) = F \circ \MK_n(t) - \MK_n \circ \mathcal{R}_n(t)=(O(t^{n+N}), \, O(t^{n+2N-1})).
\end{equation*} 
We also assume that  $n>n_0$, where $n_0$ is the integer provided by Lemma \ref{lema_contraccio_analitic}. 
We rewrite 
\begin{equation*} \label{eq_delta_analitic_2}
F \circ (\MK_n + \Delta) - (\MK_n + \Delta) \circ R = 0
\end{equation*}
in the form \eqref{eqdelta_analitic}, or using the previously defined operators,
$$
\MS_{n,\, R}\, \Delta = \MN_{n, \, F} (\Delta).
$$	
By Lemma \ref {invers_S}, if $\rho$ is small, $\MS_{n,R} $ has a right inverse and we can rewrite the equation as 
	\begin{equation*} \label{eq_final_analitic}
\Delta = \MT_{n, \, F} (\Delta).
\end{equation*}

By Lemma \ref{lema_contraccio_analitic} we have that $\MT_{n, \, 
F}$ maps $\Sigma_{n, \, N}^\alpha$ into itself and is a contraction. Then it has a unique fixed point, $\Delta^\infty \in \Sigma_{n, \, N}^\alpha$.
Note that this solution is unique once $\MK_n$ is fixed. 
Finally $K = \MK_n + \Delta^\infty$ satisfies the conditions in the statement.

The  $C^1$ character of $K$ at the origin follows from the order condition of $K$ at 0.
\end{proof}

\begin{proof}[Proof of Theorem \ref{teorema_analitic-a-posteriori}]

We write the proof for case 1, the other cases being almost identical except for some adjustments in the indices of the coefficients of $\mathcal{R}_n$. Let $n_0$ be the integer provided by Lemma \ref{lema_contraccio_analitic}. 
If the value of $n$ given in the statement is such that $n<n_0$, first we look for a better approximation $\MK_{n_0}$ of the form 
$\MK_{n_0} (t) = \hat K(t) + \sum _{j=n+1} ^{n_0} \hat K^j(t)$
with  
$\hat K^j(t) = (\hat K^x_{j} t^{j}, \hat K^y_{j+k-1} t^{j+k-1} ) $
and 
\begin{equation*} \label{forma_hatRn}
\mathcal{R}_{n_0}(t) =
\begin{cases}
\hat R(t)  &\quad  \text{ if } \;\; n\ge k+1, \\
\hat R(t) + \hat R_{2k-1}t^{2k-1} & \quad \text{ if } \;\; n\le  k. 
\end{cases}
\end{equation*}
The coefficients $\hat K^x_{j} , \hat K^y_{j+k-1}$ and 
$\hat R_{2k-1}$ are obtained imposing the condition
\begin{equation*} \label{approx-hyp-n0}
F\circ \MK_{n_0} (t) -   \MK_{n_0}\circ  \mathcal{R}_{n_0}(t) =  (O(t^{n_0+k}), O(t^{n_0+2k-1})) .
\end{equation*}
Proceeding as in Proposition \ref{prop_algoritme_1},  we obtain $\hat K^j $ 
iteratively. We denote   $ \MK_j(t) = \hat K(t) + \sum_{m=n+1} ^j \hat K^m (t) $ 
and $\mathcal{R}_j(t) = \hat R(t) + \tilde R_j(t)  $, where
$ \tilde R_j(t) = \delta _{j,k+1} \hat R_{2k-1} t^{2k-1}$.
In the iterative step we have 
\begin{equation*} \label{approx-hyp-j}
F\circ \MK_j (t) -   \MK_j\circ  \mathcal{R}_j(t) =  (O(t^{j+k}), O(t^{j+2k-1})) .
\end{equation*}
Then, 
\begin{align*}
 F(\MK_j(t) + \hat K^{j+1} (t))  - & (\MK_j+ \hat K ^{j+1})\circ (\hat R(t) + \tilde R_j(t) )  \\
= &  F(\MK_j (t)) - \MK_j (\hat R(t) ) \\
& + DF(\MK_j (t))\hat K^{j+1}(t) - \hat K^{j+1} (\hat R(t)+\tilde R_j (t)) \\
& + \int_0^1 (1-s)  D^2F(\MK_j(t) + s \hat K^{j+1}(t) ) (\hat K^{j+1}(t))^{\otimes 2} \, ds \\
&  -D \MK_j(\hat R(t) )\tilde R_j (t) \\
& - \int_0^1 (1-s)  D^2\MK_j (\hat R(t)  + s \tilde R_j(t)) (\tilde R_j(t))^2\, ds  . 
\end{align*}
The condition 
\begin{equation*} \label{approx-hyp-j+1}
F\circ  \MK_{j+1} (t) -  \MK_{j+1}\circ \mathcal{R}_{j+1}(t) =  (O(t^{j+k+1}), O(t^{j+2k})) 
\end{equation*}
leads to the same equation \eqref{sist_lineal} as in Proposition \ref{prop_algoritme_1}
which we solve in the same way. 
From this point we can proceed as in the proof of Theorem \ref{teorema_analitic} and look for $\Delta \in \MX_{n_0, \, n_0+ k-1}$ such that 
the pair $K= \MK_{n_0} + \Delta $, $R=\mathcal{R}_{n_0}$
satisfies
$ F \circ K = K \circ R$. 
We have that 
$$
K(t) - \hat K(t) = \MK_{n_0}(t) - \hat K(t) + \Delta(t)
= (O(t^{n+1}), O(t^{n+k}))  +(O(t^{n_0}), O(t^{n_0+k-1})),
$$  
with $n< n_0$.

If $n\ge n_0$ we look for $\MK^*(t) = \hat K(t)+ \hat K^{n+1}(t) $
with 
$$\hat K^{n+1}(t) = (\hat K^x_{n+1} t^{n+1}, \hat K^y_{n+k} t^{n+k} ) $$
and 
\begin{equation*} \label{forma_hatRn}
\mathcal{R}^*_n(t) =
\begin{cases}
\hat R(t)  &\quad  \text{ if } \;\; n\ge k+1, \\
\hat R(t) + \hat R_{2k-1}t^{2k-1} & \quad \text{ if } \;\; n\le k. 
\end{cases}
\end{equation*}
We determine $\hat K^x_{n+1}$, $\hat K^y_{n+k}$ so that 
$F\circ \MK^* (t) -   \MK^*\circ  \mathcal{R}^*(t) =  (O(t^{n+k+1}), O(t^{n+2k})) 
$
as in the previous case and we look for $\Delta \in \MX_{n+1, \, n+k}$ such that 
the pair $K=\MK^*+\Delta $, $R=\mathcal{R}^*$ satisfies $F\circ K = K\circ R$.
As before we obtain  
$
K(t) - \hat K(t) 
= (O(t^{n+1}), O(t^{n+k})) 
$. 
Again, the $C^1$ character of $K$ at $0$ follows form the order condition of $K$.
\end{proof}

\section{The differentiable case} \label{sec_cr}

This section is devoted to prove Theorem \ref{teorema_cr} for 
\begin{equation*}
F(x, y) = 
\begin{pmatrix}
x + c\, y \\
\quad \ \,  y
\end{pmatrix}+
\begin{pmatrix}
0 \\
p(x) + y \, q(x) +  u(x,  y) + g(x, y)
\end{pmatrix}.
\end{equation*}

As in Section \ref{sec_analytic} we use the parameterization method. To get the initial approximation we first consider the Taylor polynomial of $F$ or degree $r$ which we denote by $F^\le$ and reads
\begin{equation*}
F^\le (x, y) = 
\begin{pmatrix}
x + c\, y \\
\quad \ \,  y
\end{pmatrix}+
\begin{pmatrix}
0 \\
p(x) + y \, q(x) +  u(x,  y) 
\end{pmatrix}.
\end{equation*}

Since $F^\le $ is analytic, Theorem \ref{teorema_analitic} provides a $C^1$ map $K:[0,\rho) \to \R$, analytic on $(0,\rho)$ and a polynomial, $R$, such that 

\begin{equation} \label{analitic_cr}
F^{\leq} \circ K - K \circ R = 0 \qquad  \text{on} \quad [0,\rho).
\end{equation}

Then,  we look for $\rho>0$ and a $C^r$ function, $H= K + \Delta: (0, \, \rho) \to \R^2$, such that
\begin{equation} \label{eq_delta_cr_1}
F \circ (K + \Delta) - (K + \Delta) \circ R = 0,
\end{equation}
 In Section \ref{sec_equacio_funcional}, we establish a functional equation for $\Delta$ obtained from \eqref{eq_delta_cr_1} which will be the object of our study. In Section \ref{sec_function_spaces} we describe the function spaces where we will set such an equation and the operators  $\MS_{L, \, R}$ and $\MN_{L, \, F}$ together with their properties (Lemmas \ref{lema_S1L} and \ref{lema_NL}). Notice that although the notation of the operators is similar to the one of the operators in Section \ref{sec_analytic}, both pair of families of operators are different. 

In Section \ref{sec_contractiu} we recall the fiber contraction theorem and we also introduce the family of operators $\MT_{L, \,  F}$ given by $\MT_{L, \,  F} = \MS_{L, \, R}^{-1} \circ \MN_{L, \, F}$ and we describe its properties  in Lemmas \ref{hip_c} and \ref{hip_b}. Finally, in Section \ref{sec_dem_teorema_cr} we prove the existence of a solution of the functional equation and we conclude the proof of Theorem \ref{teorema_cr}.

\subsection{The functional equation} \label{sec_equacio_funcional}

Let $F: U \subset \R^2 \to \R^2$ be a $C^r$ map of the form \eqref{forma_normal_cr} satisfying the hypotheses of Theorem \ref{teorema_cr}.
Along the section, once having taken a $C^r$ map $F$ of the form  \eqref{forma_normal_cr}, the maps $K$ and $R$ will always refer to the analytic solutions 
of $ 
F^{\leq} \circ K - K \circ R = 0$,
on some interval $[0,\rho)$ given by Theorem \ref{teorema_analitic}.

Using \eqref{analitic_cr} and the previous notation, condition \eqref{eq_delta_cr_1} can be rewritten  as  
\begin{align}  \begin{split}\label{eqdelta_cr}
\Delta^x \circ R - \Delta^x &= c \,  \Delta^y ,  \\
\Delta^y \circ R -  \Delta^y  & = p\circ(K^x + \Delta^x) - p\circ K^x  + K^y \cdot (q\circ(K^x + \Delta^x) - q\circ K^x) \\
& \quad+ \,  \Delta^y \cdot q\circ (K^x + \Delta^x) + u \circ (K + \Delta) - u \circ K+ g\circ (K + \Delta).
\end{split}
\end{align}

Clearly, a continuous function $\Delta$ satisfies \eqref{eq_delta_cr_1} if and only if it satisfies \eqref{eqdelta_cr}. Since we want to prove differentiablity  of $\Delta$, next  we derive $r$ equations for the derivatives of $\Delta$ by formally differentiating equation \eqref{eqdelta_cr}. In our approach  we will look for continuous solutions of these equations. 

After having differentiated \eqref{eqdelta_cr} $L$ times, $1\le L\le r$, we obtain
\begin{align} 
\begin{split}\label{eqdelta_cr_Lth}
D^L\Delta^x \circ R \, (DR)^L &- D^L\Delta^x  = c \,  D^L \Delta^y + \mathcal{J}_{L, \, N}^x (\Delta, \, \dots , \,D^{L-1}\Delta ),  \\
D^L\Delta^y \circ R \, (DR)^L &- D^L\Delta^y \\
&=   p' \circ (K^x+ \Delta^x) \, D^L\Delta^x + (K^y + \Delta^y) \, q' \circ(K^x+ \Delta^x) \, D^L\Delta^x \\
&  \quad + q \circ(K^x + \Delta^x) \, D^L\Delta^y + (D u + Dg) \circ (K + \Delta) \cdot D^L\Delta   \\ & \quad   
+ \mathcal{J}_{L, \, N}^y (\Delta, \, \dots , \,D^{L-1}\Delta ),
\end{split}
\end{align}

where $\MJ_{L, \, F}^x$ and $\MJ_{L, \, F}^y$ are given by
\begin{align} \begin{split} \label{def_J}
& \MJ^x_{L, \, F} (f_0, \, \dots f_{L-1})= \Lambda^x_{L , \, R} (f_0^x, \, \dots f_{L-1}^x), \\
& \MJ^y_{L, \, F} (f_0, \, \dots f_{L-1})= \Lambda^y_{L , \, R} (f_0^y, \, \dots f_{L-1}^y) + \Omega_{L, \, F} (f_0, \, \dots, \, f_{L-1} ),
\end{split}
\end{align}
and  $\Lambda^i_{L, \, R}$,  $i= x,  y$, by
\begin{align}
\begin{split} \label{def_lambda}
\Lambda^i_{1, \, R} (f_0^i) & = 0, \\
\Lambda^i_{2, \,R} (f_0^i, \, f_1^i) & = - f_1^i \circ R  \, D^2R,  \\
\Lambda^i_{L, \,R} (f_0^i, \, \dots, \,  f_{L-1}^i) & = D [\Lambda^i_{L-1, \, R} (f_0^i,  \, \dots, \,  f_{L-2}^i )] \\
& \quad - (L-1) \, f_{L-1}^i \circ R \, (DR)^ {L-2} \, D^2R, \qquad  L \in  \{3, \, \dots, \, r \},
\end{split}
\end{align}
where in the expansion of the derivative 
$D [\Lambda^i_{L-1, \, R} (f_0^i,  \, \dots, \,  f_{L-2}^i )]$ we substitute 
$Df_i$ by $f_{i+1}$. Note that $\Lambda^i_{	L, \, R}$ does not depend on $f_0$. Moreover, $\Omega_{L, \, F}$ is given by
\begin{align}
\begin{split} \label{def_omega}
\Omega_{1, \, F} (f_0)  &= DK^x  \, (p'  \circ (K^x + f_0^x) - p' \circ K^x) + DK^y  \cdot (q \circ (K^x + f_0^x) - q \circ K^x)  \\
& \quad + K^y \cdot DK^x \,  (q' \circ (K^x + f_0^x) - q' \circ K^x) + f_0^y \cdot DK^x \, q' \circ (K^x + f_0^x)  \\
& \quad + (Du \circ (K+  f_0) - Du \circ K) DK    + Dg \circ (K + f_0) DK , \\ 
\Omega_{L, \, F} (f_0, \, &\dots, f_{L-1}) = D[ \Omega_{L-1, \, F} (f_0, \, \dots, \, f_{L-2})] + D[p' \circ (K^x + f_0^x) ] f_{L-1}^x \\
& \quad +  D[(K^y +f_0^y) q' \circ (K^x+f_0^x)]f_L^x +D[q \circ (K^x +  f_0^x)]  f_{L-1}^y  \\
& \quad+ D[(Du+Dg) \circ (K+ f_0)]\cdot f_{L-1}, \qquad \qquad L \in \{2, \, \dots , \, r\}.
\end{split}
\end{align}

Note that $\Lambda_{L , \, R} (f_0, \, \dots , \, f_{L-1} )$ comes from the differentiation on the left hand side of \eqref{eqdelta_cr} and  $\Omega_{L, \, F}(f_0, \, \dots , \, f_{L-1} )$ comes from the differentiation on the right hand side of the second equation of \eqref{eqdelta_cr}.
Expanding the derivatives in  \eqref{def_lambda} and \eqref{def_omega}
and changing  $Df_i$ by $f_{i+1}$ we obtain expressions that have to be understood as operators acting on $(f_0, \, \dots \,, f_{L-1})$, considering the $f_j$'s as independent variables. 

It is important to note that $\Lambda^i _{L ,  R}$ and
$\Omega^i_{L ,  F}$, $i=x,y$, depend in a polynomial way on $f_j$ for $j\ge 1$, but not on $f_0$.

\subsection{Function spaces, the operators $\MS_{L, \,  R}$ and $\MN_{L, \, F}$ and their properties} \label{sec_function_spaces}

We introduce next the notation and the function spaces that we will use to study the  functional equations \eqref{eqdelta_cr} and  \eqref{eqdelta_cr_Lth}.

\begin{definition}
Given $0 < \rho < 1$,  let $\MY_n$,  for $n \in \Z$, be the Banach space given by
\begin{equation*}
\MY_{n} = \{f: (0,\, \rho) \to \R \, | \, f \in C^0 (0, \, \rho), \, \|f\|_{n} := \sup_{(0, \, \rho)} \, \frac{|f(t)|}{|t|^n}  < \infty \}, 
\end{equation*}
where $C^0(0, \rho)$ denotes the space of continuous functions on $(0, \rho)$.
\end{definition}

Note that when $n\ge 1$ the functions $f$ in $\MY_n$ can be continuously extended to $t=0$ with $f(0)=0$ and, if moreover, $n\ge 2$, the derivative of $f$ can be continuously extended  to  $t=0$ with $f'(0)=0$. For $n <0$ the functions contained in $\MY_n$ may be unbounded in a neighborhood of $0$.

Note also that $\mathcal{Y}_{n+1} \subset \mathcal{Y}_n$, for all $n \in  \mathbb{Z}$. 
If $f\in \MY_m, \, g \in  \MY_n$, then $f g \in \MY_{m+n}$ and $\|fg\|_{m+n} \leq \|f\|_m \,\|g\|_n$.  If $f\in \mathcal{Y}_{n+1}$, then  $\|f\|_n \le \|f\|_{n+1} $.

Given $n, \, m \in \Z$ we denote $\mathcal{Y}_{m, \, n}: = \mathcal{Y}_m \times \mathcal{Y}_{n}$ the product space, endowed with the product norm
$$
\|f\|_{m, \, n} = \max \, \{ \|f^x\|_m, \,   \|f^y\|_n \}, \qquad f= (f^x,f^y) 
\in \mathcal{Y}_m \times \mathcal{Y}_{n}.
$$

Given $s, \, r,\, N$ positive integer numbers and $L \in \{0,  \, \dots, r \}$, we define  the spaces
\begin{equation*}
\Sigma_{L, \, N} = \prod_{j=0}^{L} \,( \MY_{s-2N+2-j, \, s-N+1-j}), \qquad  0\le L\le r
\end{equation*}
and 
\begin{equation*}
D \Sigma_{L-1, \, N} = \MY_{s-2N+2-L, \, s-N+1-L},\qquad 1\le L\le r
\end{equation*}
both endowed with the product norm.
Clearly, we have  $\Sigma_{L, \, N} = \Sigma_{L-1, \, N} \times D \Sigma_{L-1, \, N}$, and  $ \Sigma_{L,  \, N} =  \Sigma_{0, \, N} \times  \prod_{i=1}^{L}  D \Sigma_{i-1, \, N}$, for $1\le L\le r$. 

For notational convenience we also write 
$
D \Sigma_{-1,  \, N} =  \Sigma_{0,  \, N} .
$

Also, let $\alpha _i >0$, $1\le i\le r$. Given $L$ we write $\alpha = (\alpha_0, \, \dots, \, \alpha_L)$. 
We define the closed balls 
\begin{align*}
\Sigma_{0, \, N}^{\alpha_0} & =  \{f \in \Sigma_{0, \, N} \ | \ \|f\|_{\Sigma_{0, \, N}} \leq \alpha_0\}, \\
D\Sigma_{i-1, \, N}^{\alpha_i} &= \{f \in D\Sigma_{i-1, \, N} \ | \ \|f\|_{D\Sigma_{i-1, \, N}} \leq \alpha_i\}, \qquad  i\in \{1, \, \dots , r\},
\end{align*}
and the products of balls 
\begin{align*}
\Sigma_{L, \, N}^\alpha &=  \Sigma_{0, \, N}^{\alpha_0} \times  \prod_{i=1}^{L}  D \Sigma_{i-1, \, N}^{\alpha_{i}},   \qquad  L\in \{1, \, \dots , r\},
\end{align*}

For notational convenience we will write $\Sigma_{0, \, N}^{\alpha} =\Sigma_{0, \, N}^{\alpha_0}$.

An element of $\Sigma_{L, \, N}$ will be denoted by $(f_0, \, \dots, \, f_L)$, with $f_0 = (f_0^x, \, f_0^y) \in \Sigma_{0, \, N}$, and $f_i = (f_i^x, \, f_i^y) \in D\Sigma_{i-1, \, N}$, for $i =1, \, \dots, \, L $. 

For the sake of simplicity we do not write the dependence with respect to 
$r$, $s$ and $\rho$ in the notation of the previous objects.

To solve the functional equation \eqref{eq_delta_cr_1}, we look for a solution, $f_0$, of \eqref{eqdelta_cr} contained in a closed ball $\Sigma^{\alpha_0}_{0, \, N}$, and for a solution, $(f_1, \, \, \dots , \, f_L)$, of \eqref{eqdelta_cr_Lth} in a product $\Sigma^{\alpha}_{L, \, N}$, for each $L \in \{1, \, \dots, \, r\}$.
In order for the compositions in \eqref{eqdelta_cr_Lth} to be meaningful
we have to deal with $f_0$ in a ball of sufficiently small radius. Arguing as in the analytic case we take 
$\alpha_0 = \min \,  \big{\{}  \tfrac{1}{2}, \, \tfrac{d}{2} \big{\}} 
$, where $d$ is the radius of a ball contained in the domain where $F$ is $C^r$. The values of  the radii  $\alpha_i$,  $1\le i\le r$, will be determined  later (see proof of Lemma \ref{hip_b}).

In the differentiable case we consider analogous operators as in the analytical case but now we need a family of them, depending on $L$, to deal with the equations \eqref{eqdelta_cr_Lth} for the derivatives of $\Delta$. Their definitions are determined by the structure of such equations.

First, we state two auxiliary results about the iterates of $R$ and their derivatives. 

\begin{lemma} \label{lema_sector_cr}
	Let $R: [0, \, \rho) \rightarrow \R$ be a differentiable map of the form $R(t) = t +R_N t^N + O(|t|^{N+1})$, with $R_N<0$. Then, for any $\nu, \mu$ such that  $0< \nu <   (N-1)|R_N|< \mu $, there exists $\rho>0$ such that 
	\begin{equation}\label{fitaRj}
	\frac{t }{(1+ j \, \mu \, t^{N-1})^{1/N-1}}< R^j(t) < \frac{t}{(1+ j \, \nu \, t^{N-1})^{1/N-1}}, \qquad \forall \, j \geq 1,  \quad \forall \, t \in (0, \, \rho).
	\end{equation}
	As a consequence, $R$ maps $(0,  \rho)$ into itself.
\end{lemma}
If $R$ were a polynomial the upper bound in Lemma \ref{lema_sector_cr} would be an immediate corollary of Lemma \ref{lema_sector}.
\begin{proof} Let $\la >0$ and $\vp_\la (t) =\frac{t}{(1+ \la \, t^{N-1})^{1/N-1}} $ for $t\ge 0$. A computation shows that 
	$\frac{d}{dt} \vp_\la (t) =\frac{1}{(1+ \la \, t^{N-1})^{N/N-1}}>0$
	and hence $\vp_\la$ is increasing.  
	We prove \eqref{fitaRj} by induction. When $j=1$, it is easy to see that there exists $\rho >0 $ 
	such that 
	$$
	\vp_\mu(t) = \frac{t }{(1+  \mu \, t^{N-1})^{1/N-1}}< R(t) < \frac{t}{(1+ \nu \, t^{N-1})^{1/N-1}} = \vp_\nu(t), \qquad \forall \, t \in (0, \, \rho).
	$$
	Assuming \eqref{fitaRj} for $j \geq 1$, 
	\begin{align*}
	R^{j+1}(t) & = R (R^{j}(t)) < 
	\vp_\nu (R^{j}(t)) < \vp_\nu \Big(\frac{t}{(1+ j \, \nu \, t^{N-1})^{1/N-1}}\Big)\\
	& = \frac{t}{(1+ (j+1) \, \nu \, t^{N-1})^{1/N-1}}
	\end{align*}
	in the same interval $ (0, \, \rho)$. 
	The lower bound is obtained in a completely analogous way using  $\vp_\mu$.
\end{proof}

\begin{lemma} \label{lema_drj} 
	Let $R: [0, \,  \rho) \rightarrow \R$ be a differentiable map of the form $R(t) = t+R_N t^N + O(|t|^{N+1})$, with $R_N<0$, such that 
	$DR(t) = 1 + N R_N t^{N-1} + O(|t|^{N})$. For any 
	$\nu, \mu$ such that  $0< \nu <  (N-1) |R_N| < \mu $, let $\kappa =\nu/\mu$. Then, there exists $\rho>0$ such that 
	\begin{equation} \label{fitaDRj}
	DR^j(t) \leq \frac{1}{(1+ j \, \mu \, t^{N-1})^{\kappa N/N-1}}, \qquad \forall \, j \in \mathbb{N}, \quad \forall \, t \in (0, \, \rho).
	\end{equation}
\end{lemma}
\begin{proof}
	Since $N |R_N| > \nu\frac{N}{N-1}$, by the form of the derivative $DR $, 	
	there exists $\rho>0$ such that 
	$$
	0< DR(t) < 1-\frac{ \nu N}{N-1} t^{N-1}, \qquad \forall \, t \in (0,\rho).
	$$
	Using the chain rule $DR^j(t) = \Pi_{m=0}^{j-1} DR(R^m(t))$ and the lower bound in \eqref{fitaRj} we can write
	\begin{align*}
	DR^j(t) & = \exp \sum_{m=0}^{j-1} \log DR(R^m(t)) \le \exp \sum_{m=0}^{j-1} \log \Big(1-\frac{\nu N}{N-1} (R^m(t))^{N-1} \Big) \\
	& \le \exp \left(\frac{-\nu N}{N-1} \sum_{m=0}^{j-1} (R^m(t))^{N-1} \right)
	\le \exp \left(\frac{-\nu N}{N-1} \sum_{m=0}^{j-1} \frac{t^{N-1} }{(1+ m  \mu t^{N-1})} \right)\\
	&\le \exp \left(\frac{-\nu N}{N-1} \int_0^j  \frac{t^{N-1} }{(1+ s  \mu  t^{N-1})}
	\, d s\right)
	=  \exp \left(\frac{-\nu N}{\mu (N-1)} \int_0^{j\mu t^{N-1}} \frac{1 }{1+ \xi }
	\, d\xi \right)\\
	& = \exp \left(\frac{-\kappa  N}{N-1} \log (1+  j \mu t^{N-1})\right)
	= \frac{1}{(1+ j  \mu  t^{N-1})^{\kappa N/N-1}}.
	\end{align*}
\end{proof}
From now on we assume $R$ is as in the previous lemmas and
$\rho $ satisfies the conclusions of them, in particular, $R(0,\rho) \subset (0, \rho)$. 
\begin{definition} \label{def_S_cr}
	Given $L \in \{0, \, \dots , \, r\}$, let $\MS_{L, R}: D \Sigma_{L-1, N}\rightarrow D \Sigma_{L-1,  N}$ be the linear operator defined component-wise as $\MS_{L,  R} = (\MS^x_{L,  R}, \, \MS^y_{L,  R})$, with
	\begin{align*}
	\MS^x_{L, R}\, f =  \MS^y_{L, R} \, f &=  f \circ R \, (DR)^L - f.
	\end{align*} 
\end{definition}

Notice that although both components are formally identical, they act on different domains. 
\begin{definition} \label{def_N_cr}
	Given a map $F$ of class $C^r$ satisfying the hypotheses of Theorem \ref{teorema_cr},
	let $\MN_{0,  \, F} : \Sigma_{0,   \, N}^\alpha \to  \MY_{s-N+1,  \,  s}$  be the  operator given by 
	\begin{align*}
	\MN_{0,  F}^x (f_0) &= c \, f_0^y, \\
	\MN_{0,  F}^y(f_0) & =  p\circ (K^x+ f_0^x) - p \circ K^x + K^y \cdot [q \circ (K^x+f_0^x)-q \circ K^x] \\
	& \quad + f_0^y \cdot q \circ (K^x + f_0^x) + u \circ (K+ f_0) - u \circ K+g \circ (K+f_0), 
	\end{align*}
	and let $\MN_{L,  \, F} : \Sigma^\alpha_{L, \, N} \rightarrow  \MY_{s-N+1-L,  \, s-L}$, $L \in \{1, \,  \dots, \, r \}$, be the operator given by
	\begin{align*}
	\MN_{L, F}^x (f_0, \, \dots, \, f_L) &= c \, f_L^y + \MJ_{L,  N}^x(f_0, \, \dots, \, f_{L-1}), \\
	\MN_{L,  F}^y (f_0,  \, \dots, \, f_L) &= p'\circ(K^x+ f_0^x) \cdot f_L^x + (K^y + f_0^y) \cdot q'\circ(K^x+ f_0^x) \, f_L^x \\
	& \quad + q\circ(K^x + f_0^x) \cdot f_L^y + (Du + D g) \circ  (K + f_0) \cdot f_L \\
	& \quad + \mathcal{J}^y_{L, \, N}(f_0, \, \dots , \, f_{L-1}),
	\end{align*}
	where  $\MJ_{L,  N}$ are  already introduced in \eqref{def_J}, \eqref{def_lambda} and \eqref{def_omega}.
\end{definition}
	From the definition of the operators $\MS_{L,  R}$ and $\MN_{L,  F}$, the recursive expressions of $\Lambda_{L,  R}$ and $\Omega_{L,  F}$ obtained in \eqref{def_lambda} and \eqref{def_omega} and the choice of $\alpha_0$  it is clear that  the operators $\MS_{L,  R}$ and $\MN_{L,  F}$ are well defined and that $\MS_{L,  R}$ is linear and bounded.

Note that with the operators introduced above, equations \eqref{eqdelta_cr} and \eqref{eqdelta_cr_Lth} can be written now as 
\begin{equation*} 
\MS_{L, \, R} \, D^L\Delta = \MN_{L, \, F} (\Delta, \, \dots , \, D^L \Delta), \qquad (\Delta, \, \dots , \, D^L\Delta) \in \Sigma_{L,  \, N}^\alpha,
\end{equation*}
for each $L \in \{ 0, \, \dots, \, r \}$ and $\alpha_0 $ as fixed previously and some $\alpha_i >0$, $1\le i\le L$.

In the following lemmas we  prove that each of the operators $\MS_{L, \, R}$ has a bounded right inverse and we provide a bound for the norm $\|\MS_{L,  \,  R}^{-1}\|$. We also show that each of the operators $\MN_{L, \, F}$ is Lipschitz with respect to the last variable and we provide a uniform bound for the Lipschitz constant for the family $\MN_{L, F}$, $L \in \{0, \, \dots , \, r\}$. For the proofs of Lemmas \ref{lema_S1L} and \ref{lema_NL}, see Section \ref{sec_dem_lemes}.

\begin{lemma} \label{lema_S1L}
	Let $0\le L\le r$. Assume $r>k$ in case 1 and $r> 2l-1$ in cases 2 and 3.
	Then, given $0 < \nu < (N-1) |R_N| < \mu$ such that $\kappa = \nu/\mu$ satisfies $\kappa > 1/N$, there exists $\rho >0$ small enough such that, taking $(0, \, \rho)$ as the domain of the functions of $\MY_{s-N+1-L, \, s-L}$, the operator  $\MS_{L, \, R} : D \Sigma_{L-1, \,  N} \to D \Sigma_{L-1, \,  N} $  has a bounded right inverse, 
	$$
	\MS_{L,  \, R}^{-1} :  \MY_{s-N+1-L,  \,  s-L} \to D \Sigma_{L-1,  \,  N} =  \MY_{s-2N+2-L, \, s-N+1-L},
	$$
	given by 
	\begin{equation} \label{lineal_cr_1}
	\MS_{L,  R}^{  -1} \, \eta = - \, \sum_{j=0}^\infty \, \eta \circ R^j \, (DR^j)^L, \qquad \eta \in  \MY_{s-N+1-L,  \,  s-L},
	\end{equation}
	and we have the operator norm bound
	$$
	\|(\MS^x_{L,   \, R})^{-1}\| \leq  \rho^{N-1} + \tfrac{1}{\nu} \, \tfrac{N-1}{s-2N+2+L(\kappa N-1)},
	$$
	$$
	\|(\MS^y_{L, \,   R})^{-1}\| \leq  \rho^{N-1} + \tfrac{1}{\nu} \, \tfrac{N-1}{s-N+1+L(\kappa N-1)}.
	$$
\end{lemma}

\vspace{1pt}

\begin{lemma} \label{lema_NL} Let $0\le L\le r$. Assume $r>k$ in case 1 and $r> 2l-1$ in cases 2 and 3. There exists   a constant, $M >0$, for which  the family of operators $\MN_{L,   \, F}$ satisfy, for  each $L \in \{0,    \dots,   r \}$,
	$$
	\text{\emph{Lip}} \ \MN^x_{L,  \,  F}\, (f_0,\, \dots, f_{L-1}, \, \cdot) = c,
	$$
	and 
	\begin{align*}
	\text{\emph{Lip}} \ \MN^y_{L,  \,  F}\,& (f_0,\, \dots, f_{L-1}, \, \cdot) \leq   k\, |a_k| +  M  \rho,  \quad \text{(case $1$)}, \\
	\text{\emph{Lip}} \ \MN^y_{L,  \,  F}\,& (f_0,\, \dots, f_{L-1}, \, \cdot)  \\
	&\leq  \max  \{ ( (l-1) \, |K_l^y  \, b_l| + k \, |a_k|) + M  \rho, \,  |b_l| + M  \rho \}, \quad \text{(case $2$)}, \\
	\text{\emph{Lip}} \ \MN^y_{L, \,  F}\, &(f_0,\, \dots, f_{L-1}, \, \cdot) \leq  \max  \{(l-1) \, |K_l^y  \, b_l| + M \rho, \, |b_l| + M  \rho \}, \quad \text{(case $3$)},
	\end{align*}
	where $(0,   \rho)$ is the domain of the functions of $\Sigma_{L,  \,  N}^\alpha$.
\end{lemma} 

Note that the bound
we have found for Lip  $\MN_{L,   \, F} (f_0,\, \dots, f_{L-1}, \, \cdot)$ does not depend on $L$, and the obtained bounds for $\|(\MS^x_{0,   \, R})^{-1}\|$ and $\|(\MS^y_{0,  \,  R})^{-1}\|$ do not depend on $\kappa$.

\subsection{Main lemmas and the fiber contraction theorem} \label{sec_contractiu}

From $\MS_{L,\,    R}$ and $\MN_{L,   \, F}$ introduced  in Section \ref{sec_function_spaces}, we can define the operators $\MT_{L,   F}$ and $\MT^\times_{L, \,   F}$.

\begin{definition} \label{def_T} Given a map $F$ of class $C^r$ satisfying the hypotheses of Theorem \ref{teorema_cr}, let $\MT_{L,   \, F} : \Sigma_{L,   \, N}^\alpha \to D\Sigma_{L-1,   \, N}$ be the operator given by 
	$$
	\MT_{L,   \, F} = \MS_{L,  \,  R}^{-1} \circ \MN_{L,  \,  F}, \qquad L \in \{0, \, \dots , \, r\},
	$$
	and let $\MT_{L, F}^\times : \Sigma_{L,   \, N}^\alpha \longrightarrow \Sigma_{L,   \, N}$ be the operator given by 
	$$\MT_{L, \,  F}^{\times} = (\MT_{0,   \, F},\, \dots, \,\MT_{L,   \, F}), \qquad L \in \{1 , \, \dots , \,r\}.
	$$
\end{definition}

In the following results we show that, under appropriate conditions, the operators $\MT_{L, \,   F}$ have some properties strongly related to the hypotheses of the fiber contraction theorem.

\begin{lemma} \label{hip_c} Let $F$ be a  $C^r$ map satisfying the hypotheses of Theorem \ref{teorema_cr}, $\alpha_i >0$, $1\le i\le r$, and  $\alpha = (\alpha_0,\dots, \alpha _L)$, $0 \leq L \leq r$.
Then, for every $L \in \{0, \, \dots, \, r-1\}$, the operator $\MT_{L,  \,  F}: \Sigma_{L,   \, N}^\alpha \to D\Sigma_{L-1,   \, N}$ is  Lipschitz on $\Sigma_{L,   \, N}^{\alpha}$ with respect to $(f_0, \, \dots , \, f_{L-1})$, with Lipschitz constant independent of $f_L$. 
	
	Moreover, the operator
	$\MT_{r,  \,  F}: \Sigma_{r, \,  N}^\alpha \to D\Sigma_{r-1,  \,  N}$  can be decomposed as $\MT_{r,   \,  F}^{(1)} + \MT_{r, \,     F}^{(2)}$, where $\MT_{r,    \, F}^{(1)}$ is Lipschitz on $\Sigma_{r, \, N}^\alpha$ with respect to $(f_0, \, \dots , \, f_{r-1})$, with Lipschitz constant independent of $f_r$  and 
	$$
	\MT_{r,  \,   F}^{(2)} = \Big{(} 0, \, (\MS^y_{r, \, R})^{-1}  \circ (D^r g \circ (K + f_0)(DK+f_1)^{r} ) \Big{)},
	$$
	which is continuous with respect to $(f_0, f_1)$.
\end{lemma}

Next we introduce a convenient rescaling.  Given $\gamma > 0$, let 
\begin{equation} \label{reescalament}
T_\gamma(x,   y) = ( x,   \gamma \,  y). 
\end{equation}
We define $\tilde{F} = T_\gamma^{-1} \circ F \circ T_\gamma$. If $K$ and $R$ are analytic maps associated to $F$, then the corresponding analytic maps associated to $\tilde{F}$ will be given by $\tilde{K} = T_\gamma^{-1} \circ K$ and $\tilde{R} = R$. Concretely, the parameterizations of $\tilde{F}$ and $\tilde{K}$ with respect to the coefficients of $F$ and  $K$ will be given by
\begin{align*}
& \tilde{F}(x, y) = 
\begin{pmatrix}
x +  \gamma  c   y \\
\quad \ \,  y
\end{pmatrix} +
\begin{pmatrix}
0 \\
\gamma^{-1} \, a_k \, x^k +  b_l \, y \, x^{l-1} + \cdots 
\end{pmatrix}, 
\end{align*}
and
\begin{align*}
& \tilde{K} (t) = \begin{pmatrix} 
t^2 + \cdots \\
\gamma^{-1} \, K_{k+1}^y \, t^{k+1} + \cdots
\end{pmatrix},  \qquad \text{ for case $1$,} \\
& \tilde{K} (t) = \begin{pmatrix} 
t + \cdots \\
\gamma^{-1} \, K_l^y \, t^l + \cdots
\end{pmatrix}, \qquad \text{ for cases $2$ and $3$.}
\end{align*}

\vspace{1pt}
\begin{lemma} \label{hip_b}
	Given a  $C^r$ map $F$ satisfying the hypotheses of Theorem \ref{teorema_cr}, there exist $\rho_0 > 0$ and a linear transformation $T_\gamma$ as in \eqref{reescalament} such that if $\rho < \rho_0$, then the operator $\MT_{L, \, \tilde{F}} : \Sigma_{L,  \, N}^\alpha \to D\Sigma_{L-1, \,  N}$ associated to $\tilde{F} = T^{-1}_\gamma \circ F \circ T_\gamma$, for $L \in \{0, \, \dots , \, r\}$, is contractive with respect to the variable $f_L \in D\Sigma_{L-1, \, N}^{\alpha}$.
	Moreover, for a proper choice of $\alpha = (\alpha_0, \, \dots , \,  \alpha_L)$, $\MT_{L,  \, \tilde{F}}$ maps  $\Sigma_{L,  \, N}^\alpha$ into $D\Sigma_{L-1,  \, N}^{\alpha_L}$, for each $L\in \{0, \, \dots \, r\}$. 
\end{lemma}
For the proofs of Lemmas \ref{hip_c} and \ref{hip_b}, see Section \ref{sec_dem_lemes}.

\begin{remark}
	The value $\alpha_0$ denoting the radius of the ball $\Sigma_{0, \,  N}^{\alpha_0}$,  obtained previously, is forced by the definition of $\MN_{0 , \, F}$ (and thus, of $\MT_{0,  \, N}$). Indeed, since we will look for the invariant curves of $F$ as parameterizations of $\Sigma_{0, \, N}^{\alpha_0}$, their image must be contained in the domain where $F$ is $C^r$. This is not the case for the derivatives of the invariant curves, for which we do not need to put a bound on them to have the operators well defined. Also, the definition of $\MT_{L, \, F}$, for $L \in \{1, \, \dots, \, r\}$ does not force any restriction to the size of the arguments $f_1, \, \dots, \, f_L$ since the dependence with respect to these variables is polynomial.
	The  values $\alpha_1, \, \dots, \, \alpha_r$ obtained in Lemma \ref{hip_b} provide then upper bounds for the norms of the derivatives of the invariant curves of $F$. 
\end{remark}
Finally, for the convenience of the reader, we recall the fiber contraction theorem
\cite{nit71} which will be used in the proof of 
Theorem \ref{teorema_cr}. We use a version of it stated in \cite{fm00}.
\begin{theorem}[Fiber contraction theorem]
	Let $\Sigma$ and $D\Sigma$ be metric spaces, $D \Sigma$ complete, and $\Gamma : \Sigma \times D\Sigma \to  \Sigma \times D\Sigma$ a map of the form $\Gamma(\gamma,   \varphi) = (G(\gamma),   H(\gamma, \varphi))$. Assume that
	\begin{enumerate}[(a)]
		\item G has an attracting fixed point, $\gamma_\infty \in \Sigma$,
		\item $H$ is contractive with respect to the second variable, ie, for all $\gamma \in \Sigma$, $\text{\emph{Lip}} \,  H(\gamma, \cdot) \\ <1$. 
		\end{enumerate}
		Let $\varphi_\infty \in D\Sigma$ be the fixed point of $H(\gamma_\infty,   \cdot)$.
		\begin{enumerate}[(c)]
		\item $H$ is continuous with respect to $\gamma$ at $(\gamma_\infty,   \varphi_\infty)$.
	\end{enumerate}
	Then, $(\gamma_\infty,   \varphi_\infty)$ is an attracting fixed point of $\Gamma$.
\end{theorem}

\subsection{Proof of Theorem \ref{teorema_cr}.} \label{sec_dem_teorema_cr}

We give next the proof of Theorem \ref{teorema_cr}, where we use the setting and the results obtained along the previous sections.

\begin{proof}[Proof of Theorem \ref{teorema_cr}]
	Let 
	$F$ be as in the statement and  $T_\gamma$, $\gamma >0$, be defined by \eqref{reescalament}.
	It is clear that given maps $H$ and $R$, the triple $(F, \, H, \, R)$  satisfies $F \circ H = H \circ R$ if and only if $(\tilde F, \, \tilde H, \, \tilde R)$ satisfies $\tilde{F} \circ \tilde{H} = \tilde{H} \circ \tilde{R}$, where $\tilde{F} = T^{-1}_\gamma \circ F \circ T_\gamma$,  $\tilde{H} = T^{-1}_\gamma \circ H$ and $\tilde{R} = R$.  Clearly $F$ and $\tilde F$ belong to the same case $1$, $2$ or $3$ of the reduced form \eqref{forma_normal_cr}.

	To prove the theorem, we shall look for $\rho>0$ and a function $H:  (0, \, \rho) \to \R^2$, with $H(0)=0$ and $H \in C^r(0, \, \rho)$, and a map of the form $R(t) = t + R_N t^N + R_{2N-1}t^{2N-1}$, with $R_N <0$, such that
	\begin{equation} \label{general_dem}
	F \circ H = H \circ R,
	\end{equation}
	with $N=k$ for case $1$ of \eqref{forma_normal_cr} and $N=l$ for cases $2$ and $3$. 
	
	We take the value  $\gamma>0$ associated with $F$ provided in Lemma \ref{hip_b}, and we set $\tilde{F} = T^{-1}_\gamma \circ F \circ T_\gamma$.  Let $\tilde{F}^{\leq}$ be the Taylor polynomial of $\tilde{F}$ of degree $r$ at the origin. Then  it is a polynomial of the form 
	\begin{align*}
	& \tilde{F}^{\leq}(x, y) = 
	\begin{pmatrix}
	x +  \gamma \, c \, y \\
	\quad \ \,  y
	\end{pmatrix} +
	\begin{pmatrix}
	0 \\
	\gamma^{-1} \, a_k \, x^k +  b_l \, y \, x^{l-1} + h.o.t.
	\end{pmatrix}.
	\end{align*}
	
	Since we assumed $a_k>0$ for cases $1$ and $2$ and $b_l<0$ for case $3$, then by Theorem \ref{teorema_analitic}, there exists, for each case, an analytic map $\tilde{K}$ and a polynomial $R$ of the form $R(t) = t+ R_N \, t^N + R_{2N-1} \, t^{2N-1}$, with $R_N < 0$, satisfying 
	$\tilde{F}^{\leq} \circ \tilde{K} - \tilde{K} \circ R = 0$.
	
	Given such maps $\tilde{K}$ and $R$, we look for $\rho >0$ and a function $\Delta: (0, \, \rho) \to \R^2$, $\Delta \in C^r(0, \, \rho)$,
	such that 
	\begin{equation} \label{eq_delta_cr_2}
	\tilde{F} \circ (\tilde{K} + \Delta) - (\tilde{K} + \Delta) \circ R = 0.
	\end{equation}

	To do so, we consider the set of $r$ equations described in \eqref{eqdelta_cr} and \eqref{eqdelta_cr_Lth}. We take $\alpha = (\alpha_0, \, \dots, \, \alpha_r)$ with 
	$
	\alpha_0 = \min \, \big{\{} \frac{1}{2}, \, \frac{d}{2}  \big{\}}
	$,
	where $d$ is the radius of a centered ball in $\R^2$ contained in the domain where $\tilde{F}$ is of class $C^r$, and $\alpha_1, \, \dots, \, \alpha_r$ given in Lemma \ref{hip_b}. We also  take the value $\rho  >0$ associated to $\tilde{F}$ provided in Lemma \ref{hip_b}.

	Given such values of $\rho$ and $\alpha$, we take the function spaces $\Sigma_{L, \, N}^\alpha$, for $L \in \{0, \, \dots, \, r \}$, with domain $(0, \, \rho) \subset \R$.

	With the operators introduced in Definition \ref{def_T}, equation 
	\eqref{eqdelta_cr} can be written as 
	\begin{equation} \label{eq_zero}
	f_0 = \MT_{0, \, \tilde{F}} (f_0), \qquad  f_0\in \Sigma_{0, \, N}^\alpha, 
	\end{equation}
	and  each of the equations  \eqref{eqdelta_cr_Lth} can be written as 
	\begin{equation*} \label{altres_dem_cr}
	f_L = \MT_{L, \, \tilde{F}} (f_0, \, \dots , \, f_L), \qquad  (f_0, \, \dots , \, f_L) \in \Sigma_{L, \, N}^\alpha, 
	\end{equation*}
	for $ L \in  \{  1, \, \dots, \, L \}$, or equivalently, all of them together as a unique equation,
	\begin{equation}
	(f_0, \, \dots , \, f_r) = \MT_{r, \, \tilde{F}}^\times \, (f_0, \, \dots , \, f_r), \qquad (f_0, \, \dots , \, f_r) \in \Sigma_{r, \, N}^\alpha.
	\label{funcional_ultim}
	\end{equation}
	
	By Lemma \ref{hip_b} and  the Banach fixed point theorem, $\MT_{0, \,   \tilde{F}}$ has an unique attracting fixed point, $f_0^\infty \in \Sigma_{0,  N}^\alpha$, which is a solution of equation \eqref{eq_zero} and which ensures that there exists a continuous solution, $\Delta^\infty$, of \eqref{eq_delta_cr_2}. We will see next that in fact the solution $f_0^\infty$ of \eqref{eq_zero} is a function of class $C^r$.
	
	We will proceed by induction. First we prove that  $f_0^\infty $ is $C^1$.  
	
	Let us pick a $C^1$ function $f_0^0  \in \Sigma_{0, \, N}^{\alpha_0}$ such that 
	$f_1^0 := D f_0^0$ belongs to $D\Sigma_{0,  \, N}^{\alpha_1}$.
	For simplicity we take $f_0^0 = 0$. 
	Then we take the sequence $(f_0^j, \, f_1^j) = (\MT_{1,  \, \tilde{F}}^\times)^j (f_0^0, \, f_1^0)$. From the definition of the operator $\MT_{1,  \, \tilde{F}}$, we have
	\begin{equation} \label{relacio_derivades}
	D(\MT_{0, \,  \tilde{F}} (f_0^0)) = \MT_{1, \,  \tilde{F}} ( f_0^0, \, f_1^0).
	\end{equation}
	Applying \eqref{relacio_derivades} inductively we have that $f_1^j = D f_0^j$, for all $j$. Also, since $f_0^0$ is $C^1$ and $f_1^0 = D f_0^0$, all the iterates $f_0^j = (\MT_{0,  \, \tilde{F}})^j (f_0^0)$ are $C^1$, and as we have said the sequence converges in  $\Sigma^{\alpha_0}_{0, \, N}$ to $f_0^\infty$.

	Again, by Lemma \ref{hip_b}, the operator $\MT_{1, \,  \tilde{F}}:\Sigma_{0,  N}^\alpha \times D\Sigma_{0,  N}^{\alpha_1}  \to D\Sigma_{0, \,  N}^{\alpha_1}$ is contractive with respect to the variable $f_1 \in D \Sigma_{0, \, N}^\alpha$. Thus, $\MT_{1, \, \tilde{F}} (f_0^\infty, \cdot )$ has a unique attracting fixed point,  $f_1^\infty \in D \Sigma_{0,  N}^\alpha$. 
	
	Moreover, by Lemma \ref{hip_c}, $\MT_{1,\,  \tilde{F}}$ is continuous with respect to $f_0$ at any point $(f_0, \, f_1) \in \Sigma_{1, \,  N}^\alpha$. 
	Hence, by the fiber contraction theorem, $(f_0^\infty, \, f_1^\infty) \in \Sigma_{1, \,  N}^\alpha$ is an attracting fixed point of $\MT_{1, \,  \tilde{F}}^\times$, which means that the sequence $f_1^i= Df_0^j$ converges in $D \Sigma_{0,  \, N}$. That is,  $f_1^i$ converges uniformly in $C^0 (0, \, \rho)$ and therefore we have $f_1^\infty = D f_0^\infty$ and thus,  $f_0^\infty  \in C^1(0, \, \rho)$.

	Now, for every $L \in \{2, \, \dots, \, r\}$, we assume that there exists a unique attracting fixed point of $\MT_{L-1, \, \tilde{F}}^\times $, given by $(f_0^\infty, \, \dots, \, f_{L-1}^\infty) \in \Sigma_{L-1, \,  N}^\alpha$, such that $f_0^\infty \in C^{L-1} (0, \, \rho)$ and 
	$$
	f_1^\infty = D f_0^\infty, \, \dots , \, f_{L-1}^\infty = D^{L-1} f_0^\infty.
	$$
	
	We will see next that in fact $f_0^\infty$ is of class $C^L$.
	
	Let us pick again the function $f_0^0=0\in C^{L} (0, \, \rho)$, and let us take also $f_1^0 := D f_0^0, \, \dots , \, f_L^0 := D^L f_0^0$. Then we have $(f_0^0, \, \dots f_{L-1}^0) \in \Sigma_{L-1, \, N}^\alpha$ and $f_L^0 \in D\Sigma_{L-1,  \, N}^{\alpha_L}$.
	
	From the definition of the operator $\MT_{L, \, \tilde{F}}$, we have
	\begin{equation} \label{relacio_derivades_L}
	D(\MT_{L-1, \,  \tilde{F}} (f_0^0, \, \dots, \, f_{L-1}^0)) = \MT_{L, \, \tilde F} ( f_0^0, \, \dots , \,  f_L^0).
	\end{equation}
	
	Then let $(f_0^j, \, \dots , \, f_L^j) = (\MT_{L,  \tilde F}^\times)^j (f_0^0, \, \dots , \, f_L^0)$. Applying \eqref{relacio_derivades_L} inductively we have $f_1^j = D f_0^j, \, \dots , \, f_L^j = D^L f_0^j$, for all $j$, and then the iterates	$(f_0^j, \, \dots, \, f_{L-1}^j) = (\MT^\times_{{L-1},  \, N})^j (f_0^0, \, \dots, \, f_{L-1}^0)$ 
	are such that $f^j_m \in C^{L-m}$, for $m \in \{0, \, \dots, \, L-1\}$. By the induction hypothesis, the sequence
	$(f_0^j, \, \dots, \, f_{L-1}^j)$ converges in $\Sigma_{L-1,  \, N}$ to the solution $(f_0^\infty, \, \dots, \, f_{L-1}^\infty)$ and
	$$
	f_1^\infty = D f_0^\infty, \, \dots , \, f_{L-1}^\infty = D^{L-1} f_0^\infty.
	$$
	
	Also, applying Lemmas \ref{hip_c} and \ref{hip_b} and the fiber contraction theorem, the sequence $f_L^j= D^Lf_0^j$ converges in $D \Sigma_{L-1, \,   N}$. That is,  $f_L^j$ converges uniformly in $C^0 (0, \, \rho)$ and therefore we have $f_L^\infty = D^L f_0^\infty$ and thus,  $f_0^\infty  \in C^L(0, \, \rho)$.
	In conclusion $f_0^\infty\in C^r(0,\rho)$.  
	
	Finally, the $C^r$ map  $\tilde{H} = \tilde{K} + \Delta$ with $\Delta = f^\infty_0$ parameterizes the stable manifold of $\tilde{F}$ 
	and therefore it is $C^r$.
	
When $F$ is $C^\infty$, to see that the stable manifold is $C^\infty$ we take $r_1$ satisfying the hypotheses of the theorem and $r_2 > r_1$. The previous proof provides $H_1= K_{r_1} + \Delta_1$ and $H_2= K_{r_2} + \Delta_2$ defined in $(0, \rho_1)$ and $(0, \rho_2)$ and of class $C^{r_1}$ and $C^{r_2}$  respectively that parameterize stable manifolds $W_1$ and $W_2$.  Theorem 4.1 of \cite{fontich99}, which is proved by geometric methods, provides the uniqueness of the stable manifold in this setting.  If $\rho_2< \rho_1$, since we deal with stable manifolds  we can extend $W_2$ iterating by $F^{-1}$ to recover $W_1$. Then $W_1$ is $C^{r_2}$ for all $r_2 > r_1$. 
\end{proof}

\section{Proofs of the technical results} \label{sec_dem_lemes}

We will give detailed proofs for Lemmas \ref{lema_S1L}, \ref{lema_NL}, \ref{hip_c} and \ref{hip_b}, which correspond to the differentiable case. Lemmas \ref{invers_S}, \ref{lema_N_analitic} and \ref{lema_contraccio_analitic} are simplified complex versions of Lemmas  \ref{lema_S1L}, \ref{lema_NL} and \ref{hip_b} respectively.

\subsection{Properties of the operators $\MS_{L, \, R}$ and $\MN_{L ,  \, F}$}

\begin{proof}[Proof of Lemma \ref{lema_S1L}]
	A simple computation shows that the expression $\eqref{lineal_cr_1}$ of $\MS_{L,   \, R}$ formally satisfies $\MS_{L, \,   R} \circ (\MS_{L,  \, R})^{-1}  \, \eta = \eta$, for $\eta \in \MY_{s-N+1-L, \, s-L}$.
	
	We give the details of the proof for the second component $\MS^y_{L, \,   R}:  \MY_{s-N+1-L} \to  \MY_{s-N+1-L}$ of the operator $\MS_{L,  \,  R}$, the details for  $\MS^x_{L,  \,  R}:  \MY_{s-2N+2- L} \to  \MY_{s-2N+2-L}$ being completely analogous. The results for $\MS_{L,  \,  R}$ follow immediately because the components of the operator are uncoupled.
	
	We take $ \kappa > 1/N$  and $\mu, \nu $ such that  $0< \nu <  (N-1)| R_N |  < \mu$ and $\nu/\mu =\kappa  $. By Lemmas \ref{lema_sector_cr} and \ref{lema_drj} there exists  $\rho>0$  such that $R$ maps the interval $(0,   \rho)$ into itself and the bounds
	\eqref{fitaRj} and \eqref{fitaDRj} hold.  
	Then, given $\eta \in  \MY_{s-L}$
	\begin{align*}
	|(\eta \circ R^{j}  (DR^j)^L) (t)| & \leq \,  \|\eta \|_{s-L} \, |R^j (t)|^{s-L} \, |DR^j (t)|^L \\
	& \leq  \|\eta \|_{s-L} \, \frac{ t^{s-L}}{(1+j \, \nu \,  t^{N-1})^{\tfrac{s-L}{N-1}}} \, \frac{1}{(1+j \, \mu \,  t^{N-1})^{\tfrac{\kappa NL}{N-1}}} \\
	& \leq\,  M\, \|\eta \|_{s-L} \, \frac{1}{j^{\tfrac{s+L(\kappa N-1)} {N-1}}}, \qquad \qquad \forall \,  t \in (0, \, \rho),
	\end{align*}
	hence, since $s\ge r\ge N > N-1$,  \eqref{lineal_cr_1} converges uniformly on $(0,   \rho)$ by the Weierstrass $M$-test. Thus, $(\MS^y_{L,  \, R})^{-1} \, \eta = - \sum_{j=0}^\infty \eta \circ R^j (DR^j)^L$ is continuous on $(0,  \rho)$. 
	
	Now, we prove that $(\MS^y_{L,  \, R})^{-1}$ is a bounded operator from  $ \MY_{s - L}$ to $ \MY_{s -N+1-L}$ and we  obtain a bound for its norm.  
	Again, having chosen $\kappa = \nu/\mu $, from Lemmas \ref{lema_sector_cr} and \ref{lema_drj} one has,  
	\begin{align*}
	\|(\MS^y_{L,  \, R})^{-1} \,\eta & \|_{s-N+1-L}
	\leq  \sup_{t \in (0,   \rho)} \,  \frac{1}{ t^{s-N+1-L}} \,  \sum_{j=0}^{\infty} |\eta (R^j (t)) (DR^j(t))^L| \\
	& \leq  \|\eta \|_{s-L} \, \sup_{t \in (0,   \rho)} \,  \frac{1}{ t^{s-N+1-L}} \, \sum_{j=0}^{\infty} \,  \frac{ t^{s-L}}{(1+j   \nu    t^{N-1})^{\tfrac{s-L}{N-1}}} \, \frac{1}{(1+j   \mu    t^{N-1})^{\tfrac{\kappa NL}{N-1}}} ,
	\end{align*}
	and, bounding the sum by an appropriate integral, we obtain the bound 
	\begin{align*}
	\frac{1}{ t^{s-N+1-L}} \,& \sum_{j=0}^{\infty} \,  \frac{ t^{s-L}}{(1+j   \nu    t^{N-1})^{\tfrac{s-L}{N-1}}} \, \frac{1}{(1+j   \mu    t^{N-1})^{\tfrac{\kappa NL}{N-1}}}  \\
	& \leq   t^{N-1}    \bigg{(} \, 1 + \int_0^{\infty} \frac{1}{(1+x  \nu    t^{N-1})^{\frac{s-L+\kappa NL}{N-1}}} \, dx \, \bigg{)} \\
	&=  t^{N-1} + \frac{1}{\nu} \, \frac{N-1}{s-N+1+L(\kappa N-1)}. 
	\end{align*}
	Therefore, we get
	\begin{align*}
	\|(\MS^y_{L, \,  R})^{-1} \, & \eta\|_{s-N+1-L} 
	\\ &\leq \|\eta\|_{s-L} \,\sup_{t \in (0,   \rho)} \bigg{(}  t^{N-1} + \frac{1}{\nu} \, \frac{N-1}{s-N+1+L(\kappa N-1)} \bigg{)},
	\qquad \eta \, \in \, \mathcal{X}_{s-L} ,
	\end{align*}
	which shows that $(\MS^y_{L,  \, R})^{-1}:  \MY_{s-L} \to  \MY_{s-N+1-L}$ is bounded and
	$$\|(\MS^y_{L,  R})^{-1}\| \leq  \rho^{N-1} + \frac{1}{\nu} \, \frac{N-1}{s-N+1+L(\kappa N-1)}.
	$$
	In the same way, $(\MS^x_{L,  \, R})^{-1}:  \MY_{s-N+1-L} \to  \MY_{s-2N+2-L}$ is bounded and 
	$$\|(\MS^x_{L,  \, R})^{-1}\| \leq  \rho^{N-1} + \frac{1}{\nu} \, \frac{N-1}{s-2N+2+L(\kappa N-1)}.
	$$
\end{proof}
\begin{proof}[Proof of Lemma \ref{invers_S}]
	The operators 
	$\MS^y_{n,  \, R}$ do not contain the term $(DR)^L$ so that the proof is similar to the one  of Lemma \ref{lema_S1L} with $L=0$. However, the domain of the  functions in the spaces $\MX_n $ is the complex sector $S(\beta, \rho) $, and therefore in this case we have to apply Lemma \ref{lema_sector}. Notice that in this case we do not need lower bounds for $R^j(z)$. 
\end{proof}

\begin{proof}[Proof of Lemma \ref{lema_NL}]
	To distinguish the roles of the variables $(f_0, \, \dots , \, f_{L-1})$
and $f_L$ we will denote the latter by $h_L$. The statement concerning the component $\MN^x_{L,   F}$ is clear by the definition of $\MN_{L,   F}$.
	
	For $\MN^y_{L,   N}$ we first deal with the case $L=0$.

	Since $g(x,   y) = o(\|(x,   y)\|^r)$ and $g\in C^r $ we have
	$
	D_i g (x,  y)  = o(\|(x,   y)\|^{r-1}),\; i = 1,   2.
	$
	
	For every $h_{0},   \tilde{h}_{0} \in \Sigma_{0,   N}^\alpha$,
	from the definition of the operator $\MN^y_{0,   F}$,	one can write 
	\begin{align*}
	\begin{split} \label{taylor_lema}
	& \MN^y_{0,   \, F} (h_{0}) - \MN^y_{0,  \,  F} (\tilde{h}_{0})  \\
	&= \Big{(}\int_0^1 p'\circ (K^x + \tilde {h}_0^x  + s  (h_0^x - \tilde{h}_0^x)) \, ds \\
	&\quad + (K^y+ {h}_0^y)  \int_0^1 q'\circ (K^x + \tilde{h}_0^x + s  (h_0^x - \tilde{h}_0^x)) \, ds  \\
	& \quad + \int_0^1 (D_1u +D_1 g) \circ (K + \tilde{h}_0+ s  (h_0 - \tilde{h}_0)) \, ds \Big{)} \, (h_0^x- \tilde{h}_0^x) \\
	&\quad + \Big{(} q \circ (K^x + \tilde{h}_0^x) +  \int_0^1 (D_2u +D_2 g) \circ (K + \tilde{h}_0+ s  (h_0 - \tilde{h}_0)) \, ds \Big{)} \, (h_0^y- \tilde{h}_0^y).
	\end{split}
	\end{align*}
	Let us denote, for $ s \in [0,   1]$  
	\begin{align*}
	\xi_s & = \xi_s (h_0,   \tilde{h}_0) = K + \tilde{h}_0 + s  (h_0 - \tilde{h}_0), \\
	\varphi &= \varphi(h_0,   \tilde{h}_0)= \int_0^1 p'\circ \xi_s^x \, ds + (K^y+ {h}_0^y) \, \int_0^1 q'\circ \xi_s^x \, ds + \int_0^1 (D_1u+D_1g) \circ \xi_s \, ds, \\
	\psi &= \psi (h_0,   \tilde{h}_0) = q\circ (K^x + \tilde{h}_0^x) +  \int_0^1 (D_2u+D_2g) \circ  \xi_s \, ds,
	\end{align*}
	so that we have 
	\begin{equation}\label{fitalipL=0}
	\| \MN^y_{0,   F} (h_{0}) - \MN^y_{0,   F} (\tilde{h}_{0}) \|_{s} \leq \| \varphi (h_0,   \tilde{h}_0)    (h_0^x- \tilde{h}_0^x)\|_s + \|\psi(h_0,   \tilde{h}_0)   (h_0^y- \tilde{h}_0^y) \|_s.
	\end{equation}
	
	For case $1$ we have $K \in   \MY_{2,   \, k+1}$ and, since $s=2r$ and $r>k$, then  for every  $h_{0},   \tilde{h}_0 \in \Sigma_{0,   \, k}^\alpha$ we have $(h_{0},  \tilde{h}_0) \in \MY_{4,   \, k+2} $.
	Thus we can bound the norm
	$$
	\|\xi^x_s\|_2 = \sup_{t \in (0,   \rho)} \, \frac{1}{ t^2} \, |K^x(t) +  \tilde{h}^x_{0}(t) + s( h^x_{0}(t)- \tilde{h}^x_{0}(t))| \leq 1 + M  \rho,
	$$
	for all $s \in [0,   1]$.
	
	Moreover, checking the orders of $\varphi$ and $\psi$, taking into account the properties of $p$, $q$, $u$ and $g$, we have 
	$$
	\varphi \in  \MY_{2k-2}, \qquad \psi \in  \MY_{k} \subset  \MY_{k-1},   \qquad \forall \; h_0,   \tilde{h}_0 \in \Sigma_{0,  \,  k}^\alpha.
	$$
	More precisely, we can bound
	\begin{align}
	\begin{split} \label{bounds_case1}
	\|\varphi\|_{2k-2} & \leq  \sup_{s \in [0,   1]} \,( \|p' \circ \xi_s^x\|_{2k-2}+ \|(K^y+{h}_0^y) \, q' \circ \xi_s^x+ D_1g\circ \xi_s+  D_1u\circ \xi_s \|_{2k-2}) \\
	&  \leq \sup_{s \in [0,   1]} \, \sup_{t \in (0,    \rho)} \, \frac{1}{ t^{2k-2}} \,(k\, |a_k| |\xi_s^x(t)|^{k-1} +  M\,  t^{2k-1} ) \\
	& \leq k  |a_k| + M   \rho ,
	\end{split} \\
	\|\psi\|_{k-1} & \leq M\rho,  
	\end{align}
	for all $h_0,   \tilde{h}_0 \in \Sigma_{0,   k}^{\alpha_0}$.

	Then, from \eqref{fitalipL=0} we have
	\begin{align*}
	\|\MN^y_{0,   F} (h_0)- \MN^y_{0,   F} (\tilde{h}_0) \|_{s} \leq   &   \|\varphi\|_{2k-2}\,\|h_0^x - \tilde{h}_0^x\|_{s-2k+2} +  \|\psi \|_{k-1}\,\|h_0^y - \tilde{h}_0^y\|_{s-k+1} \\
	\leq  &  ( k  |a_k| + M   \rho) \|h_0^x - \tilde{h}_0^x\|_{s-2k+2} + \rho \,M\,\|h_0^y - \tilde{h}_0^y\|_{s-k+1},
	\end{align*}

	which proves that $
	\text{Lip} \ \MN^y_{0,  F} \leq  k\, |a_k|  + M \rho,
	$
	for case $1$.
	
	For cases $2$ and $3$ the bounds for $\text{Lip} \ \MN^y_{0,  \, F}$
	are obtained in an analogous way.
	In these cases we have $K \in  \MY_{1,  \, l}$ and we obtain  $\xi_s \in  \MY_{2,  \, l+1}$. Take $h_{0}, \, \tilde{h}_0 \in \Sigma_{0, l}^\alpha$. Since $r> 2l-1$,
	$$
	\varphi \in  \MY_{2l-2}, \qquad \psi \in  \MY_{l-1},
	$$
	with the following bounds for their norms, 
	\begin{equation} \label{bounds_case2}
	\|\varphi\|_{2l-2} \leq k\, |a_k| + (l-1) |K_l^y \, b_l| + M \rho, \qquad  \|\psi\|_{l-1} \leq  |b_l| + M \rho,
	\end{equation}
	in case $2$ and 
	\begin{equation} \label{bounds_case3}
	\|\varphi\|_{2l-2} \leq  (l-1) |K_l^y \, b_l| + M \rho, \qquad  \|\psi\|_{l-1} \leq  |b_l| + M \rho,
	\end{equation}
	in case $3$.
	
	The proof for $L \geq 1$ is similar.
	Given $f_0, \, \dots, \, f_{L-1}$ and $h_L,   \tilde{h}_L \in D\Sigma_{L-1,   \, N}^{}$, from the definition of $\MN^y_{L,   \, N}$, we have		
	\begin{align*}
	\MN^y_{L,   \, F} \, (f_0,\, \dots, f_{L-1}, \, h_L) &-  \MN^y_{L,   \, F} \, (f_0,\, \dots, f_{L-1}, \, \tilde{h}_L) \\
	& = \big{(} p'\circ(K^x+ f_0^x)  + (K^y + f_0^y) \, q'\circ(K^x+ f_0^x) \\
	& \quad + (D_1u +D_1 g) \circ(K + f_0) \big{)} (h_{L}^x -\tilde{h}_{L}^x) \\
	& \quad +\big{(} q \circ(K^x + f_0^x)  + (D_2u +D_2 g) \circ (K + f_0)  \big{)} (h_L^y - \tilde{h}_L^y).
	\end{align*}
	Given $f_0 \in \Sigma_{0,   N}^\alpha$, we denote  
	\begin{align*}
	\tilde{\varphi} &= \tilde{\varphi} (f_0) = p' \circ(K^x+ f_0^x)  + (K^y + f_0^y) \, q' \circ(K^x+ f_0^x) + (D_1u +D_1 g) \circ (K + f_0),\\
	\tilde{\psi} & =  \tilde{\psi}(f_0) = q \circ (K^x + f_0^x)  +(D_2u + D_2 g) \circ (K + f_0),
	\end{align*}
	so that we can write
	\begin{align*}
	\|\MN^y_{L,   F} \, (f_0,\, \dots, f_{L-1}, \, h_L) -  \MN^y_{L,   F} \, & (f_0,\, \dots, f_{L-1}, \, \tilde{h}_L) \|_s  \\
	& \leq \|\tilde{\varphi} (f_0)  (h_L^x- \tilde{h}_L^x)\|_s + \|\tilde{\psi}(f_0) (h_L^y- \tilde{h}_L^y) \|_s.
	\end{align*}
	
	The orders of $\tilde\varphi$ and $\tilde\psi$ are the same as the ones of the corresponding $\varphi$ and $\psi$ when $L=0$, respectively, for each of the cases 1, 2 and 3. That is, 
	$$
	\tilde{\varphi} \in  \MY_{2k-2}, \qquad  \tilde{\psi} \in  \MY_{k} \subset  \MY_{k-1}, 
	$$
	for case $1$ and 
	$$
	\tilde \varphi \in  \MY_{2l-2}, \qquad \tilde \psi \in  \MY_{l-1}, 
	$$
	for cases $2$ and $3$.
	As in the case $L=0$, for each $f_0 \in \Sigma_{0,   N}^{\alpha_0}$, the order of $K + f_0$ is the same as the one of $K$. Therefore we get the same  bounds for the norms of $\tilde\varphi$ and $\tilde\psi$, namely those obtained in \eqref{bounds_case1} - \eqref{bounds_case3}, and finally the bounds in the statement.
\end{proof}
\begin{proof}[Proof of Lemma \ref{lema_N_analitic}] 
The proof is completely analogous to the proof of Lemma \ref{lema_NL} in the case $L=0$, the only difference being that here the functions in the spaces $\MX_n$ are defined in sectors $S(\beta, \rho)$ instead of the interval $(0, \rho)$.
\end{proof}

\subsection{Proofs of Lemmas \ref{hip_c} and \ref{hip_b}}

\begin{proof}[Proof of Lemma \ref{hip_c}]
As before, to distinguish the roles of the variables $f_L$ and $(f_0, \, \dots , \, f_{L-1})$  we will denote the former by $h_L$. Since $\MT_{L,  \, F}= \MS_{L, \, R}^{-1} \circ \MN_{L,  F}$ and $\MS_{L,  R}^{-1}$ is linear and bounded, along the proof we will deal only with $\MN_{L,  \, F}$.
	
Given a function $h_L \in D\Sigma_{L-1, N}^{\alpha_L}$ we decompose 
	$$
\MN_{L,  F} (f_0, \, \dots , \, f_{L-1}, h_L) = \mathcal{A}_{h_L, \, F} (f_0) + \MJ_{L,  F} (f_0, \, \dots , \, f_{L-1}),
$$
where 
$\mathcal{A}_{h_L, \, F}
:=(\mathcal{A}^x_{h_L, \, F} ,\mathcal{A}^y_{h_L, \, F} ) 
: \Sigma_{0,  N}^\alpha \to  \MY_{s-N+1-L,  s-L}$ is  the auxiliary operator 
	\begin{align*}
	\mathcal{A}^x_{h_L, \, F} (f_0)   = & \,  c \, h_L^y, \\
	\mathcal{A}^x_{h_L, \, F} (f_0)   = & \, p'\circ(K^x+ f_0^x) \cdot h_L^x + (K^y + f_0^y) \cdot q'\circ(K^x+ f_0^x) \, h_L^x \\
	& + q\circ(K^x + f_0^x) \cdot h_L^y + (Du+D g) \circ  (K + f_0) \cdot h_L,
	\end{align*}

and we will work on $\mathcal{A}_{h_L, \, F}$ and $\MJ_{L,  F}$ separately.
	
Clearly   $ \mathcal{A}^x_{h_L, \, F}$ is uniformly Lipschitz on $\Sigma_{0,  \,  N}^{\alpha}$. To deal with 
$\mathcal{A}^y_{h_L, \, F}$, let $f_0, \, \tilde{f}_0 \in \Sigma_{0,  \, N}^\alpha$. Then 
	\begin{align*}
	\mathcal{A}^y_{h_L, \, F} (f_0) -  \mathcal{A}^y_{h_L, \, F} (\tilde{f}_0) & = \varphi_{h_L}  (f_0, \, \tilde{f}_0) (f_0^x - \tilde{f}_0^x) 
	+ \psi_{h_L} (f_0, \, \tilde{f}_0)  (f_0^y - \tilde{f}_0^y) \\
	&  \quad + \theta(f_0, \, \tilde{f}_0)   (f_0 - \tilde{f}_0) \cdot h_L,
	\end{align*}
	with
	\begin{align*}
	\varphi_{h_L} & = {\varphi}_{h_L}  (f_0, \tilde{f}_0)   = h_L^x \int_0^1 p'' \circ (K^x + \tilde{f}_0^x + s(f_0^x - \tilde{f}_0^x)) \, ds \\
	& \quad + h_L^y \, \int_0^1 q'\circ (K^x + \tilde{f}_0^x + s(f_0^x - \tilde{f}_0^x)) \, ds \\
	& \quad + h_L^x \, (K^y + f_0^y) \int_0^1 q'' \circ (K^x + \tilde{f}_0^x + s(f_0^x - \tilde{f}_0^x)) \, ds, \\
	\psi_{h_L} & = {\psi}_{h_L} (f_0, \tilde{f}_0)  = h_L^x \,  q' \circ (K^x + \tilde f_0^x), \\
	\theta & = {\theta}  (f_0, \tilde{f}_0) =  \int_0^1 (D^2u+D^2g) \circ (K+ \tilde{f_0} + s(f_0 - \tilde{f}_0)) \, ds .
	\end{align*}
First we deal with case 1.
	By similar arguments as in  Lemma \ref{lema_NL} we have  $\varphi_{h_L}  (f_0, \, \tilde{f}_0)  \\ \in  \MY_{2r-2-L}  \subseteq \MY_{2k-2-L}$, $ \psi_{h_L} (f_0, \, \tilde{f}_0) \in  \MY_{r-1-L}  \subseteq \MY_{k-1-L} $. All the entries of the matrix $\theta  (f_0, \, \tilde{f}_0)$  belong to  $ \MY_{0}$. Also, it is clear that  the quantities $ \|\varphi_{h_L} \, (f_0, \, \tilde{f}_0)\|_{2k-2-L}$, $\|\psi \,  (f_0, \, \tilde{f}_0)\|_{k-1-L}$ and the  $\|\cdot \|_{ \MY_{0}}$-norm of the entries of  $\theta \, (f_0, \, \tilde{f}_0) $ are uniformly bounded for  $f_0, \, \tilde{f}_0 \in \Sigma_{0,  N}^\alpha$, the norm depending on $\alpha _0$ in the form $\rho^m \alpha_0$ for some $m>0$ and depending linearly on $\alpha_L$.
	
	Then, since $h_L$ is fixed, we get 
	\begin{align*}
	\|\mathcal{A}^y_{h_L, \, F} (f_0) - \mathcal{A}^y_{h_L, \, F} (\tilde{f}_0)\|_{2r-L} & \leq \|\varphi_{h_L} \, (f_0, \, \tilde{f}_0)\|_{2k-2-L} \, \|f_0^x - \tilde{f}_0^x\|_{2r-2k+2} \\
	& \quad + \|\psi_{h_L} \,  (f_0, \, \tilde{f}_0)\|_{k-1-L} \, \|f_0^y -  \tilde{f}_0^y\|_{2r-k+1} \\
	& \quad + M \|h_L\|_{D \Sigma_{L-1,  k}} \, \|f_0-  \tilde{f}_0\|_{\Sigma_{0,  k}} \\
	& \le M \,\alpha_L \|f_0 - \tilde{f}_0\|_{\Sigma_{0,  k}}.
	\end{align*}
	
 Similarly we also obtain 
 $	\|\mathcal{A}^y_{h_L, \, F} (f_0) - \mathcal{A}^y_{h_L, \, F} (\tilde{f}_0)\|_{2r-L}
 \le M \,\alpha_L \|f_0 - \tilde{f}_0\|_{\Sigma_{0,  k}}$
 for cases $2$ and $3$, where in these cases we have $\varphi_{h_L} \in  \MY_{r-L-1}$, $ \psi_{h_L} \in  \MY_{r-l-L}$ and the entries of $\theta $ belong to  $\MY_{0}$.  This proves that $\mathcal{A}_{h_L, \, F}$ is uniformly Lipschitz on $\Sigma_{0, \, N}^\alpha$.

	Next we deal with $\MJ_{L,  \, F}$. Recall that we have, for every $L \in \{1, \, \dots , r\}$,
	\begin{align*} 
	& \MJ^x_{L, \,  F} (f_0, \, \dots , \, f_{L-1})= \Lambda^x_{L ,  R} (f_0^x, \, \dots , \, f_{L-1}^x), \\
	& \MJ^y_{L, F} (f_0, \, \dots , \, f_{L-1})= \Lambda^y_{L , \,  R} (f_0^y, \, \dots , \, f_{L-1}^y) + \Omega_{L, \, F} (f_0, \, \dots , \, f_{L-1}),
	\end{align*}
	where $\Lambda^x_{L, R}$ and $\Lambda^y_{L,  \, R}$ are given recursively in \eqref{def_lambda} and $\Omega_{L, \, F}$ is given recursively in \eqref{def_omega}.
	
	From \eqref{def_lambda}, $\Lambda^i_{1 , R}=0$ and, for $L\ge 2$, one can check by induction that
\begin{equation} \label{formulaLambda}
	\Lambda^i_{L , \, R} (f_1^i, \, \dots , \, f_{L-1}^i) = \sum_{j=1}^{L-1} \, P_{L, \, j} \,  f_j^i \circ R , \qquad i= x, \, y,
\end{equation}
	where each function $P_{L, \, j}$ is a polynomial on the variable $t$. 
	
Indeed, $P_{2, \, 1}(t) = -D^2R(t) \in \MY_{N-2}$. Assuming \eqref{formulaLambda} and applying the recurrence \eqref{def_lambda} we have
\begin{align*}
 \Lambda^i_{L +1,\,  R}  &= \sum_{j=1}^{L-1} \, P'_{L, \, j} \, f_j^i \circ R
+   \sum_{j=1}^{L-1} \, P_{L, \, j} \, f_{j+1}^i \circ R \, DR 
-L \,  f_{L}^i \circ R \, (DR)^{L-1} D^2R \\
& =  
 P'_{L, \, 1} \, f_1^i \circ R
+ \sum_{j=2}^{L-1} \,\big( P'_{L, \, j} 
+    P_{L, \, j-1}DR \big)  f_{j}^i \circ R 
\\
& \quad + \big( P_{L, \, L-1}DR -L  (DR)^{L-1} D^2R \big) f_{L}^i \circ R.
\end{align*} 	 
We also have the recurrences 
\begin{align*}
& P_{L+1, \, 1} (t) = P'_{L, \, 1} (t), \\
& P_{L+1, \, j} (t) = P'_{L, \, j} + P_{L, \, j-1} DR, \qquad 2 \leq j \leq L-1, \\
& P_{L+1, \, L} (t) = P_{L, \, L-1} DR- L(DR)^{L-1} D^2R,
\end{align*}
and then we also deduce by induction that $P_{L, \, j} =\MY _{N+j-1-L}$.

From this, it is clear that $\Lambda_{L,  R} = (\Lambda^x_{L, R}, \, \Lambda^y_{L, R}): \Sigma_{L-1,  N}^\alpha \to  \MY_{s-N+1-L,  s-L}$ is linear and bounded, so it is uniformly Lipschitz in $\Sigma_{L-1,  \, N}^\alpha$. 

Also, from \eqref{def_omega}, one can see that $\Omega_{L, \, F}$ is a polynomial operator on the variables $f_1 ,\, \dots , \, f_{L-1} $ having coefficients depending on $f_0$.

When $L=1$, 
\begin{align*} 
	\Omega_{1, \, F} (f_0) - \Omega&_{1, \, F} (\tilde f_0)  \\
	= & DK^x  \int_0^1 (p'' \circ (K^x + \tilde f_0^x + s( f_0^x-\tilde f_0^x))\, ds\, ( f_0^x-\tilde f_0^x) \\ 
	& + DK^y  \int_0^1 (q' \circ (K^x + \tilde f_0^x + s( f_0^x-\tilde f_0^x))\, ds \, ( f_0^x-\tilde f_0^x) \\
	& + (K^y + \tilde{f}_0^y) \,  DK^x \int_0^1 q'' \circ (K^x + \tilde{f}_0^x + s(f_0^x - \tilde{f}_0^x)  ) \, ds \, ( f_0^x-\tilde f_0^x) \\
	&  + ( f_0^y-\tilde f_0^y) \,  DK^x  \, q' \circ (K^x +  f_0^x) \\
	& + DK^x  \,\int_0^1 (D^2u+D^2 g) \circ (K + \tilde f_0 + s( f_0-\tilde f_0))\, ds \,( f_0-\tilde f_0)
\end{align*}
and hence there exists $M>0$ depending on $F$ and $\alpha _0$ such that 
$$
\|	\Omega_{1, \, F} (f_0)  - \Omega_{1, \, F} (f_0)\|_{s-1}
\le M \| f_0-\tilde f_0\|_{\Sigma_{0,  N}}.
$$

For $L>1$, we decompose $\Omega_{L, \, F} = \Omega_{L, \, F}^{(1)}  + \Omega_{L, \, F}^{(2)}$, where 
$$
\Omega_{L, \, F}^{(2)} = \Omega_{L, \, F}^{(2)}(f_0, f_1) = D^Lg \circ (K+f_0)(DK+ f_1)^L, 
$$
and  $\Omega_{L, \, F}^{(1)}  = \Omega_{L, \, F} - \Omega_{L, \, F}^{(2)} $. The difference $\Omega_{L, \, F}^{(1)} (f_0, \dots , f_{L-1})  - \Omega_{L, \, F}^{(1)} (\tilde f_0, \dots ,\tilde f_{L-1})$  is a sum of terms of the form $c_L(f_0, \tilde f_0) \Pi$, where  $\Pi $ is a product of factors among $f^{x,y}_j$, 
$\tilde f^{x,y}_j$ and $f^{x,y}_j -\tilde f^{x,y}_j$ and such that $c_L(f_0, \tilde f_0) \Pi\in \MY_{s-L}$. 
From \eqref{def_omega} we estimate  $\Omega_{L, \, F}^{(1)} (f_0, \dots ,f_{L-1})  - \Omega_{L, \, F}^{(1)} (\tilde f_0, \dots , \tilde f_{L-1})$ iteratively, where a part of it comes from 
$$D[\Omega_{L-1, \, F}^{(1)} (f_0, \dots ,f_{L-2})  - \Omega_{L-1, \, F}^{(1)} (\tilde f_0, \dots , \tilde f_{L-2})].$$ When one differenciates formally the terms  $c_{L-1}(f_0) \Pi$, the new terms $c_{L-1}(f_0, \, \tilde{f_0})' \Pi$ and $c_{L-1}(f_0, \, \tilde{f_0}) \Pi'$ appear. 

The factors of each function $c_L(f_0, \, \tilde{f_0})$ are derivatives of $K^i$, $f^i_j$, $\tilde f^i_j$,  $\int_0^1 (Q_1 (K^i + \tilde f_0^i + s( f_0^i-\tilde f_0^i))\, ds \,( f_0^x-\tilde f_0^x) $ and $(Q_2 (K^i +  f_0^i) $,  where $Q_1$, $Q_2$ are polynomials (derivatives of $p$, $q$ or $u$), and the derivative of
\begin{equation}\label{termedmg}
\int_0^1 D^m g \circ (K + \tilde f_0 + s( f_0-\tilde f_0))\, ds \,( f_0-\tilde f_0), \qquad m \leq L-1.
\end{equation}
 When taking a derivative, each term generates several terms, each one having bigger order, the same order or the same order minus one unit. 
The term $\Omega_{L, \, F}^{(2)}$ is Lipschitz when $L<r$. When $L=r$, it is continuous (in the given topology) since $D^r g$ isuniformly continuous in closed balls. 

On the other hand, when taking a derivative to $\Pi$ we obtain terms which have the same factors except one which is transformed to its derivative, that is, $f^{x,y}_j$ is transformed to  $f^{x,y}_{j+1}$ or  $f^{x,y}_j-\tilde f^{x,y}_j$ is transformed to  
$f^{x,y}_{j+1}-\tilde f^{x,y}_{j+1}$. In any case the order decreases by one unit so we have that their $\|\cdot \|_{s-L}$-norm is bounded by 
$M_L \| (f_0, \dots ,f_{L-1})  -  (\tilde f_0, \dots \tilde f_{L-1})\|_{\Sigma_{L,N}}  $, where the constant $M_L $ depends on $\alpha _0, \dots, \alpha_L$ and $F$ but not on the $(f^i_j)'s$.                          
\end{proof}

\begin{proof}[Proof of Lemma \ref{hip_b}]
	By its  definition,  the operator $\MT_{L,  F}$ satisfies
	\begin{align}
	\begin{split} \label{formula_lip_T}
	 \text{Lip} \;  \MT_{L,  \, F} (f_0, \, \dots, \, f_{l-1},  \cdot  ) \leq \max \{ & \|(\MS^x_{L,  R})^{-1}\| \, \text{Lip} \; \MN^x_{L, \,  F}(f_0, \, \dots, \, f_{L-1}, \cdot), \\
	 & \|(\MS^y_{L,  R})^{-1}\| \, \text{Lip} \; \MN^y_{L,  F}(f_0, \, \dots, \, f_{L-1}, \cdot) \}.
	\end{split}
	\end{align}
	From the estimates obtained in Lemmas \ref{lema_S1L} and  \ref{lema_NL} we have that the bounds of $\text{Lip} \; \MN_{L, \,  {F}}$ $(f_0, \, \dots, \, f_{L-1}, \cdot)$ do not depend on $L$, and taking $\kappa <1$ close to $1$  the obtained bounds for $\|\MS_{L,  \, R}^{-1}\|$ decrease as $L$ increases, so that it holds
	$$
	\text{Lip} \ \MT_{L, \,  {F}} (f_0, \, \dots, \, f_{L-1}, \cdot) \leq  \text{Lip} \ \MT_{0,  {F}} (f_0, \, \dots, \, f_{L-1}, \cdot), \qquad \forall \;  L \in \{0, \, \dots, \, r\}.
	$$
	
	Actually, this inequality is for the obtained bounds for the Lipschitz constants of the family $\{ \MT_{L,  \, F}  \}_{L}$.  Note also that $	\text{Lip} \ \MT_{0, \,  {F}} (f_0, \, \dots, \, f_{L-1}, \cdot)$ does not depend on $\kappa$.
	
	To prove the first part of the lemma we will find an appropriate map $T_\gamma$ given in \eqref{reescalament}  (that is, an appropriate  value for $\gamma$) such that if the coefficients of $F$ satisfy the  hypotheses of Theorem \ref{teorema_cr}, then the corresponding operator $\MT_{L, \,  \tilde{F}}$ associated to $\tilde{F} = T_\gamma^{-1} \circ F \circ T_\gamma$ satisfies $\text{Lip} \ \MT_{L, \, \tilde{F}} (f_0, \, \dots, \, f_{L-1}, \cdot) <1$.

	We start by considering case $1$. From \eqref{formula_lip_T} and the estimates obtained in Lemmas \ref{lema_S1L}  and \ref{lema_NL}, given $\nu \in (0, \, (k-1) |R_k|)$ there is $\tilde\rho_0$ such that for $\rho < \tilde\rho_0$ we have the bound
	\begin{align*}
	\text{Lip} \;  \MT_{L, \,  \tilde{F}} (f_0, \, \dots, \, f_{L-1}, \cdot)  \leq \max \,  \big{\{}  & \big{(} \rho^{k-1} + \frac{1}{\nu} \frac{k-1}{2r-2k+2}\big{)} \, \gamma \,  |c| ,  \\
	& \big{(} \rho^{k-1}+ \frac{1}{\nu} \frac{k-1}{2r-k+1} \big{)} \, (  \gamma^ {-1} \, k \, a_k + M \, \rho ) \big{\}}.
	\end{align*}
	
	Clearly, the condition
	\begin{equation} \label{lip_k_cr}
	\max \, \big{\{}   \gamma \, \frac{|c|}{|R_k|} \, \frac{1}{2r-2k+2}, \, \gamma^{-1} \,  \frac{k\, a_k}{|R_k|} \, \frac{1}{2r-k+1}   \big{\}} <1,
	\end{equation}
	is sufficient  to ensure that there exists  $0 < \rho_0 < \tilde{\rho}_0$  such that $\text{Lip} \, \MT_{L,  \tilde{F}}(f_0, \, \dots, \, f_{l-1}, \cdot) \\ < 1$ for $\rho<\rho_0$,  since keeping $\kappa$ fixed one can choose a value for $\nu$ close enough to $(k-1)|R_k|$.

	Then, taking $\gamma = \sqrt{\frac{k \, a_k}{c}\frac{2r-2k+2}{2r-k+1}}$, condition \eqref{lip_k_cr} is given by
	$$
	\frac{2k(k+1)}{(2r-2k+2)(2r-k+1)} <1,
	$$
	which holds for any $k \geq 2$ and $r \geq \frac{3}{2} \, k$. Hence, if $r \geq \frac{3}{2} \, k$, the operator $\MT_{L, \,  \tilde{F}}$ associated to $\tilde{F} =  T_\gamma^{-1} \circ F \circ T_\gamma$ for the  chosen value  of  $\gamma$ satisfies $\text{Lip} \ \MT_{L, \,  \tilde{F}} (f_0, \, \dots, \, f_{L-1}, \cdot) <1$, for every $L  \in \{0, \, \dots , \, r\}$, provided that $\rho < \rho_0$.
	
	For cases $2$ and $3$ of the reduced form of $F$ the result follows in a similar way choosing an appropriate value for the parameter $\gamma$.
	
For case $2$  we have,  from \eqref{formula_lip_T} and the estimates obtained in Lemmas \ref{lema_S1L}  and \ref{lema_NL},
	that the condition
	\begin{equation} \label{lip_l2_cr}
	\max \big{\{}  \gamma \, \frac{|c|}{|R_l|} \frac{1}{r-2l+2}, \, \gamma^{-1} \,\frac{(l-1) \, |K_l^y \, b_l| + k \, a_k}{|R_l|} \frac{1}{r-l+1}, \,   \frac{|b_l|}{|R_l|} \frac{1}{r-l+1}  \big{\}} <1,
	\end{equation}
	is sufficient  to ensure that there exists $\rho_0>0$ such that $\text{Lip} \, \MT_{L,  F}(f_0, \, \dots, \, f_{l-1}, \cdot)< 1$ for $\rho<\rho_0$.

	Then, taking $\gamma = \sqrt{\frac{(l-1)  \, |K_l^y \, b_l| + k \, a_k}{c} \, \frac{r-2l+2}{r-l+1} }$, condition \eqref{lip_l2_cr} is given by
	$$
	\max \big{\{} \frac{\beta}{(r-2l+2)(r-l+1)}  \big{(} (l-1) + \frac{c \, k\, a_k }{b_l^2} \beta  \big{)} ,  \, \frac{ \beta}{r-l+1} \big{\}} <1,
	$$
	where $\beta = \frac{2l \, |b_l|}{|b_l - \sqrt{b_l^2 + 4\, c \, a_k \,  l}|}$,
	which is the condition for $F$ assumed for case $2$.

	For case $3$ we have, again from \eqref{formula_lip_T} and the estimates obtained in Lemmas \ref{lema_S1L}  and \ref{lema_NL},
	that the condition
	\begin{equation} \label{lip_l3_cr}
	\max \big{\{}  \gamma \, \frac{|c|}{|R_l|} \frac{1}{r-2l+2}, \, \gamma^{-1} \,\frac{(l-1) \, |K_l^y \, b_l|}{|R_l|} \frac{1}{r-l+1}, \,   \frac{|b_l|}{|R_l|} \frac{1}{r-l+1}  \big{\}} <1,
	\end{equation}
	
	is sufficient  to ensure that there exists $\rho_0>0$ such that $\text{Lip} \, \MT_{L,  F}(f_0, \, \dots, \, f_{l-1}, \cdot)< 1$ for $\rho<\rho_0$.
	
	Taking $\gamma = \frac{|b_l|}{|c|} \sqrt{\frac{(l-1)(r-2l+2)}{l(r-l+1)}}$, condition \eqref{lip_l3_cr} is given by
	$$
	\max \big{\{} \frac{l(l-1)}{(r-2l+2)(r-l+1)}, \, \frac{l}{r-l+1}   \big{\}} <1,
	$$
	that is, 
	$$
	\frac{l(l-1)}{(r-2l+2)(r-l+1)} <1,
	$$
	which is the condition for $F$ assumed for case $3$.

Finally we prove that given a map $F$ satisfying the hypotheses of Theorem \ref{teorema_cr} such that the associated operators $\MT_{L, \, F}$ satisfy $\text{Lip} \, \MT_{L, \,  F}(f_0, \, \dots, \, f_{L-1}, \cdot)< 1$ for $\rho<\rho_0$,  one can find a new $ \rho_0$, maybe smaller than the previous one, and a choice for the values $\alpha_1, \, \dots, \,  \alpha_r$
	such that, if $\rho < \rho_0$, then $\MT_{L,  \, F}$
	maps $\Sigma_{L,  \, N}^\alpha$ into $D\Sigma_{L-1,  \, N}^{\alpha_L}$, for every $L \in \{0, \, \dots , \, r\}$. 
	
		For later use, we estimate $\|\MT_{L, 	\,  F} (0, \, \dots, \, 0)\|_{D\Sigma_{L-1,  \, N}}$.
	From Definition \ref{def_N_cr} of $\MN_{L, \,  F}$ and the definition of $ \MJ_{L, \,  F}$ in \eqref{def_J} we have 
	$$
	\MN_{L, \, F} (0, \, \dots , \, 0) = \MJ_{L, \, F}(0, \, \dots, \, 0) = (0, \,  D^L (g \circ K)).
	$$
	Moreover $D^L ( g \circ K)  (t) = o (|t|^{s-L})$. Therefore, for every $\varepsilon > 0$, there is $\rho_0>0$ such that if $\rho< \rho_0$, then 
	\begin{align}
	\begin{split} \label{cota_TL}
	\|\MT_{L,  \, F} (0, \, \dots, \, 0)\|_{D\Sigma_{L-1,  N}}  & \leq \|(\MS^y_{L,  \, R})^{-1}\| \, \|\MN^y_{L,  \, F} (0, \, \dots , \, 0)\|_{s-N+1-L, \, s-L} \\
	& \leq \|(\MS_{L,  \, R}^y)^{-1} \| \, \sup_{t \in (0, \, \rho)} \, \frac{|D^L ( g\circ K)(t)|}{t^{s-L}} \le \|(\MS_{L, \, R}^y)^{-1}\| \, \varepsilon.
	\end{split}
	\end{align}

	Next we proceed by induction. For $L=0$, we have, for all $f_0 \in \Sigma_{0,  N}^{\alpha_0}$,
	\begin{align*}
	\|\MT_{0,  \, F}  (f_0) \|_{\Sigma_{0,  \, N}}  \leq \| \MT_{0, \, F}  (f_0) -  \MT_{0,  F}  (0)\|_{\Sigma_{0,  \, N}}  + \|\MT_{0, \, F}  (0)\|_{ \Sigma_{0, \, N}} \\
	\leq  \alpha_0  \, \text{Lip} \; \MT_{0,  \, F} + \|\MT_{0, \,  F}  (0)\|_{ \Sigma_{0,  \, N}}.
	\end{align*}

	We need to see then that there exists $\rho_0>0$ such that $\|\MT_{0, \, F}  (f_0)\|_{\Sigma_{0,  \, N}} \leq \alpha_0$ provided that $\rho < \rho_0$. Clearly this holds from the estimate obtained in \eqref{cota_TL} since we have  $\text{Lip} \; \MT_{0, \,  F} <1$, and then one can take  $\rho_0$ 
	such that
	$\alpha_0 \,  \text{Lip} \; \MT_{0, \, F}  + \|\MT_{0,  F} (0) \|_{\Sigma_{0,  \, N}} \leq \alpha_0 $ for $\rho<\rho_0$. Hence we have  $\MT_{0, \, F} (\Sigma_{0,  \, N}^{\alpha_0}) \subseteq \Sigma_{0, \,  N}^{\alpha_0}$.
	
	Now, we take $\rho_1< \rho_0$ and  we denote by $\varepsilon_L$   the quantity
	$$
	\varepsilon_L = \|\MT_{L,  \, F} (0, \, \dots, \, 0)\|_{D\Sigma_{L-1, \,  N}},
	\qquad L \in \{1, \, \dots, \, r\},
	$$
	taking as the domain of the functions of $\Sigma^\alpha_{L,  \, N}$ the interval $(0, \, \rho_1)$.
	
	Continuing with the induction procedure,  for each $L \in \{ 1, \, \dots, \, r \}$, we  decompose
	\begin{align}
	\begin{split} \label{acotacio_TL}
	\|\MT_{L,  F}  (f_0,  \dots ,  f_L) \|_{D \Sigma_{L-1,   N}} & \leq \| \MT_{L,  \, F}  (f_0,  \dots ,  f_L) -  \MT_{L,  F}  (f_0,  \dots ,  f_{L-1},  0)\|_{D\Sigma_{L-1,  N}}  \\
	& \quad + \| \MT_{L,   F}  (f_0,  \dots ,  f_{L-1},  0)  -  \MT_{L,  F}  (0,  \dots ,  0) \|_{D \Sigma_{L-1,   N}}\\
	& \quad + \|\MT_{L,  F}  (0,  \dots ,  0)\|_{D \Sigma_{L-1,  N}} .
	\end{split}
	\end{align}
	
	Also, from the definitions of $\MT_{L, \,  F}$ and $\MN_{L,  \, F}$ we have
	$$
	\MT_{L,   F} (f_0,  \dots ,  f_{L-1},  0) = \MS_{L,  R}^{-1} \circ \MN_{L,  F} (f_0,  \dots ,  f_{L-1}, \, 0) = \MS_{L,   R}^{-1} \circ \MJ_{L, F} (f_0,  \dots ,  f_{L-1}) .
	$$
	
	Now we have to consider separately the cases $L < r$ and $L=r$. For $L < r$ we have, from Lemma \ref{hip_c}, that $\MT_{L,F} (f_0, \, \dots , \, f_L)$ is uniformly Lipschitz with respect to $(f_0, \, \dots , \, f_{L-1})$ in $\Sigma^\alpha_{L,  N}$, and in particular,
	$$
	\text{Lip} \; \MT_{L, F} (\cdot, 0) = \text{Lip} \; (\MS^{-1}_{L, R} \circ \MJ_{L, \,  F}). 
	$$
	
	Therefore, from \eqref{acotacio_TL} we have
	\begin{align}
	\begin{split}  \label{acotacio_T2}
	&\|\MT_{L, \,  F}  (f_0, \, \dots , \,  f_L) \|_{D \Sigma_{L-1, \,  N}}  \leq \text{Lip} \; \MT_{L, \, F} (f_0, \, \dots , \, f_{L-1}, \cdot)  \, \|f_L\|_{D \Sigma_{L-1, \,  N}} \\
	&  \quad + \text{Lip} \; (\MS^{-1}_{L, \, R} \circ \MJ_{L, \,  F}) \, \|(f_0, \, \dots , \, f_{L-1})\|_{\Sigma_{L-1,  \, N}}   +  \|\MT_{L,  \, F}  (0, \, \dots , \, 0)\|_{D \Sigma_{L-1,  \, N}} \\
	& \leq \alpha_L \, \text{Lip} \; \MT_{L, \,  F} (f_0, \, \dots , \, f_{L-1}, \cdot)  + 
	\max \, \{\alpha_0, \, \dots, \, \alpha_{L-1}  \} \, \text{Lip} \; (\MS^{-1}_{L,\, R} \circ \MJ_{L, \,  F}) + \varepsilon_L.
	\end{split}
	\end{align}
	
	Then we  can choose a value  for the radius $\alpha_L$ of  $D\Sigma_{L-1,  \, N}^{\alpha_{L}}$ to ensure that $\MT_{L,\,   F}$ maps $\Sigma^\alpha_{L,  \, N}$ into $D\Sigma^{\alpha_L}_{L-1, \, N}$. Since we have 
	$	\text{Lip} \; \MT_{L, F} (f_0, \, \dots, \, f_{L-1}, \, \cdot) <1$, then taking 
	$$
	\alpha_L  = \frac{\varepsilon_L + \text{Lip} \; (\MS^{-1}_{L, \, R} \circ \MJ_{L,  F}) \, 	\max \, \{\alpha_0, \, \dots, \, \alpha_{L-1}  \}  }{ 1- \text{Lip} \; \MT_{L,  \, F} (f_0, \, \dots , \, f_{L-1}, \cdot) },
	$$
	we have, applying \eqref{acotacio_T2}, 
	$$
	\|\MT_{L, \,  F}  (f_0, \, \dots , \,  f_L) \|_{D \Sigma_{L-1, \,  N}} \leq \alpha_L,
	$$
	for each $(f_0, \, \dots, \, f_L) \in \Sigma^\alpha_{L, \,  N}$, as we wanted to see.  
	
	For $L=r$ we proceed in an analogous way, except for the fact that we use the decomposition $\MT_{r, \, F}^{(1)} +  \MT_{r, \, F}^{(2)}$ given in Lemma \ref{hip_c}. Since $\MT_{r, \, F}^{(1)}$ is Lipschitz with respect to $(f_0, \, \dots , \, f_r)$, its contribution is as in the cases $L <r$. As we also have 
	$
	\MT_{r,  \, F} (f_0, \, \dots , \, f_{L-1}, \, 0)  = \MS_{r,  \, R}^{-1} \circ \MJ_{r, \,  F} (f_0, \, \dots , \, f_{r-1})
	$ and $\MS_{r, \, R}^{-1}$ is linear, we can  denote $
	\MT_{r,\,  F}^{(i)} (f_0, \, \dots , \, f_{r-1}, \, 0)  =  \MS_{r,  \, R}^{-1} \circ \MJ_{r, F}^{(i)} (f_0, \, \dots , \, f_{r-1})
	$, for $i = 1, 2$, with $\MJ_{r, \, F}^{(2)}(f_0, \, \dots , \, f_{r-1}) =  (0, D^r g \circ (K+f_0) (DK+f_1)^{r} ) $.
	
	We proceed as in \eqref{acotacio_TL}, but now for the second term of the sum we have, applying Lemma \ref{hip_c}, 
	\begin{align*}
	  \| \MT_{r, \,  F}  (f_0, \, \dots , \, f_{r-1}, \, 0)  -  &\MT_{r, \, F}  (0, \, \dots , \, 0) \|_{D \Sigma_{r-1, \,  N}}\\
	& \leq  \text{Lip} \; (\MS^{-1}_{r, \, R} \circ \MJ_{r, \,  F}^{(1)}) \, \|(f_0, \, \dots , \, f_{r-1})\|_{\Sigma_{r-1,  \, N}}  \\
	& \quad + \|(\MS^y_{r, \, R})^{-1}\| \| D^r g \circ (K+f_0) (DK+f_1)^{r}  \|_{s-r}.
	\end{align*} 
	To bound the quantity $\| D^r g \circ (K+f_0)  (DK+f_1)^{r} \|_{s-r}$, note that we have $D^r g(x, \, y) \\= o(\|(x, \, y)\|^0)$. 
	
	For case $1$ of the reduced form of $F$ we have $(DK+f_1)^{r}  \in \MY_{r}$ and thus, for every $\varepsilon >0$ there is $\rho_0$ such that if $\rho < \rho_0$, then
	$$
 \| D^r g \circ (K+f_0)  (DK+f_1)^{r} \|_{r} = \sup_{t \in (0, \rho)}  \frac{1}{t^r} |D^r g \circ (K+f_0)(t)  (DK+f_1)^{r} (t) | < \varepsilon.
	$$
	Similarly, for cases $2$ and $3$ we have $(DK+f_1)^{r} \in \MY_{0}$
	and 
	$$
	\| D^r g \circ (K+f_0)  (DK+f_1)^{r} \|_{0} = \sup_{t \in (0, \rho)}   |D^r g \circ (K+f_0)(t)  (DK+f_1)^{r} (t)| < \varepsilon.
	$$
	Then, for the chosen radius $\rho_1$ we denote $\hat \varepsilon = \| D^r g \circ (K+f_0)  (DK+f_1)^{r} \|_{s-r} $ and similarly as in  \eqref{acotacio_T2} we have 
		\begin{align*}
	& \|\MT_{r, \,  F}  (f_0, \, \dots , \,  f_r) \|_{D \Sigma_{r-1, \,  N}}  \\
	& \leq \alpha_r \, \text{Lip} \; \MT_{r, \,  F} (f_0, \, \dots , \, f_{r-1}, \cdot)  + 
	\max \, \{\alpha_0, \, \dots, \, \alpha_{r-1}  \} \, \text{Lip} \; (\MS_{r,\, R} \circ \MJ_{r, \,  F}^{(1)}) + \varepsilon_r + \hat \varepsilon,
	\end{align*}
	and therefore the statement of the lemma follows choosing 
		$$
	\alpha_r  = \frac{\hat \varepsilon + \varepsilon_r + \text{Lip} \; (\MS^{-1}_{r,  \, R} \circ \MJ_{r,  F}^{(1)}) \, 	\max \, \{\alpha_0, \, \dots, \, \alpha_{r-1}  \}  }{ 1- \text{Lip} \; \MT_{r,  \, F} (f_0, \, \dots , \, f_{r-1}, \cdot) }.
	$$
\end{proof}

\begin{proof}[Proof of Lemma \ref{lema_contraccio_analitic}]
	The proof is a simplified version of the one of Lemma \ref{hip_b}, here the functions in the spaces $\MX_n$ being defined in sectors $S(\beta, \rho)$ instead of the interval $(0, \rho)$. To prove that $\MT_{n, \, N}$ is contractive for $n > n_0$ we proceed as in Lemma \ref{hip_b} for $L=0$, but here the index $n$ appears in the denominator of the bound obtained for Lip$\; \MT_{n, \, N}$, proving that the operator is contractive for $n$ large enough. The second part of the lemma if proved also as in  Lemma \ref{hip_b} for $L=0$. In this case we have 
	$\MT_{n, \, N}(0) = \MS_{n, \, N}^{-1} \, \ME_n$, and thus $\|\MT_{n, \, N}(0)\|_{\Sigma_{n, \, N}}< \varepsilon$ for $\rho$ sufficiently small.
\end{proof}

\section{Conclusions}

In this paper, we have considered two-dimensional maps with a parabolic fixed point with non-diagonalizable linear part in both the analytic and differentiable cases. We have considered three different cases depending on the nonlinear terms (the generic maps are contained in case $1$).

 In the analytic case, we have proved the existence of an analytic one-dimensional invariant manifold (away from the fixed point) under suitable conditions on some of the coefficients of the nonlinear terms of the map. The existence of an analytic manifold in such case was already proved in \cite{fontich99} using a variation of McGehee's method. However, here we have used the parameterization method, which provides approximations of the manifolds up to any order, and we have also presented an \emph{a posteriori} result. 

In the $C^r$ case, first we have used our results for analytic maps applied to the Taylor polynomial of degree $r$ of the map. In this way we have obtained an analytic invariant manifold which is used as an approximation to apply the parameterization method to the original map. Moreover, we have applied the fiber contraction theorem to obtain the differentiability result. Concretely, we have proved that if the regularity of the map is bigger than some (easily computable) value, then there exists an invariant manifold of the same regularity, away from the fixed point. 

This is in contrast with the results in \cite{zhang16}, where in our case $1$, it is proved that the manifolds are of class $C^{[(k+1)/2]}$, where $k$ is as in Section \ref{sec:notation}, even if they use the Poincar\'e normal form instead of our reduced form. In that case, the obtained regularity also holds at the fixed point. 

Throughout the paper, the results are stated for stable curves.  However, we showed that the same results hold true for unstable curves and that they can be obtained directly from the stated results, without having to invert explicitly the given map.

As in the analytic case,  for the $C^r$ case we provided approximations of the invariant manifolds up to an order that depends on the regularity of the map. 

We remark that, from the computational point of view, since the dynamics on the invariant manifold close to the fixed point is extremely slow, it is important to have good approximations of the manifold in a not so small neighborhood of the fixed point in order to be able to globalize the local manifold with a reasonably small number of iterations.

\section*{Acknowledgements}

C. Cufí-Cabré has been partially supported by the Spanish Government grants MTM2016-77278-P (MINECO/ FEDER, UE), PID2019-104658GB-I00 (MICINN/FEDER, UE)  and BES-2017-081570, and by the Catalan Government grant 2017-SGR-1617.
E. Fontich has been partially supported by the Spanish Government grants MTM2016-80117-P (MINECO/ FEDER, UE) and  PID2019-104851GB-I00 (MICINN/FEDER, UE) and by the Catalan Government grant 2017-SGR-1374.

\end{document}